\documentclass{amsart}

\usepackage[T1]{fontenc}
\usepackage{pgf,pgfarrows,pgfnodes,pgfautomata,pgfheaps,pgfshade,hyperref, amssymb}
\usepackage{amssymb}

\usepackage[english]{babel}
\usepackage[capitalize]{cleveref}

\makeatletter
\def\namedlabel#1#2{\begingroup
    #2%
    \def\@currentlabel{#2}%
    \phantomsection\label{#1}\endgroup
}
\makeatother

\usepackage{mathtools}
\usepackage[colorinlistoftodos]{todonotes}
\usepackage{soul}
\usepackage[utf8]{inputenc} 
\usepackage[T1]{fontenc}
\usepackage[normalem]{ulem}
\usepackage{cancel}

\usepackage{tikz}
\usetikzlibrary{arrows}

\usepackage{tikz}
\usetikzlibrary{arrows}

\usepackage{xcolor}

\usepackage{tabularx}
\usepackage{array, ltablex, multirow}
\usepackage{makecell}
\setcellgapes{6pt}
\makegapedcells

\usepackage{faktor}
\usepackage{xfrac}

\usepackage{enumerate}
\usepackage{comment}

\def\R{{\mathbb R}}
\def\Z{{\mathbb Z}}

\def\Q{{\mathbb Q}}
\def\N{{\mathbb N}}

\def\F{{\mathbb F}}

\newcommand*\diff{\mathop{}\!\mathrm{d}}

\hypersetup{
    colorlinks,
    linkcolor={blue!30!black},
    citecolor={green!50!black},
    urlcolor={blue!80!black}
} 

\usepackage{mathrsfs} 
\usepackage{dsfont}

\usepackage{float}

\usepackage{graphicx} 

\usepackage{pgfplots}
\pgfplotsset{width=10cm,compat=1.9}

\newtheorem{theorem}{Theorem}[section]
\newtheorem{proposition}[theorem]{Proposition}

\newtheorem{lemma}[theorem]{Lemma}

\newcounter{thmcounter}

\newcounter{introthmcounter}

\newtheorem{corollary}[theorem]{Corollary}
\theoremstyle{definition}
\newtheorem{definition}[theorem]{Definition}
\newtheorem*{definition*}{Definition}

\newtheorem{question}[theorem]{Question}
\newtheorem*{question*}{Question}
\newcounter{proofcount}
\AtBeginEnvironment{proof}{\stepcounter{proofcount}}

\theoremstyle{remark}
\newtheorem*{remark*}{Remark}

\makeatletter
\def\section{\@startsection{section}{1}%
  \z@{.5\linespacing\@plus.7\linespacing}
{.8\baselineskip}%
  {\normalfont\fontsize{11}{13}\centering\bfseries}%
}

\makeatletter
\def\subsection{\@startsection{subsection}{2}%
  \z@{.4\linespacing\@plus.7\linespacing}
{.6\baselineskip}%
  {\normalfont\centering\bfseries}%
}

\numberwithin{equation}{section}

\makeatletter                  
\@addtoreset{claim}{proofcount}
\makeatother

\theoremstyle{remark}

\newtheorem{remark}[theorem]{Remark}

\usepackage{pgfarrows}

\renewcommand{\d}{~\mathrm{d}}
\def\N{{\mathbb N}}

\newcommand{\xmt}{(X,\mu,T)}
\newcommand{\zmr}{(Z,m,R)}
\newcommand{\B}{\mathscr{B}}
\newcommand{\yns}{(Y,\nu,S)}

\newcommand{\gen}{\texttt{\textup{gen}}}
\newcommand{\wt}{\widetilde}

\newcommand{\Erdos}{Erd\H{o}s}
\newcommand{\Folner}{F\o{}lner}

\newcommand{\oh}{{\rm o}}

\makeatletter
\renewcommand{\tocsection}[3]{%
  \indentlabel{\@ifnotempty{#2}{\bfseries\ignorespaces#1 #2\quad}}\bfseries#3}
\renewcommand{\tocsubsection}[3]{%
  \indentlabel{\@ifnotempty{#2}{\ignorespaces#1 #2\quad}}#3}
\makeatother

\title{Infinite unrestricted sumsets in subsets of abelian groups with large density}
\author{Dimitrios Charamaras}
\address{Centre de recherche math\'ematiques\\
Universit\'e de Montr\'eal, Montr\'eal, QC, Canada}
\email{dimitrios.charamaras@umontreal.ca}

\author{Ioannis Kousek}
\address{Mathematics Institute, University of Warwick, Coventry, UK}
\email{ioannis.kousek@warwick.ac.uk}

\author{Andreas Mountakis}
\address{Department of Mathematics, Northwestern University, Evanston, IL, USA}
\email{andreas.mountakis@northwestern.edu}

\author{Trist{\'a}n Radi{\'c}}
\address{Department of Mathematics, Northwestern University, Evanston, IL, USA}
\email{tristan.radic@u.northwestern.edu}

\date{\today}


\begin{document}

\begin{abstract}
Let $(G,+)$ be a countable abelian group such that the subgroup $\{g+g\colon g\in G\}$ has finite index and the doubling map $g\mapsto g+g$ has finite kernel. We establish lower bounds on the upper density of a set $A\subset G$ with respect to an appropriate \Folner{} sequence, so that $A$ contains a sumset of the form $\{t+b_1+b_2\colon b_1,b_2\in B\}$ or $\{b_1+b_2\colon b_1,b_2\in B\}$, for some infinite $B\subset G$ and some $t\in G$. Both assumptions on $G$ are necessary for our results to be true. We also characterize the \Folner{} 
sequences for which this is possible.
Finally, we show that our lower bounds are optimal in a strong sense. 
\end{abstract}

\maketitle

\tableofcontents

\section{Introduction}
In 1974, Hindman proved a central result in Ramsey theory: for any finite partition of the natural numbers $\N = C_1 \cup \cdots \cup C_r$, there is an infinite set $B \subset \N$ such that all the finite sums of distinct elements of $B$ are contained in a single cell $C_i$ (see \cite{Hindman74}). This result was extensively generalized to any semigroup 
(see \cite[Corollary 5.9]{hindman_strauss1998algebra}).
More recently, Kra, Moreira, Richter and Robertson established a density version of Hindman's theorem \cite{kmrr_density_finite_sums} (see also \cite{hernandez_kousek_radic2025density_Hindman}). A special case of their result was conjectured by Erd\H{o}s  \cite{erdos1973} and proved earlier by the same authors in \cite{kmrr2}.

\begin{theorem}{\cite[Theorem 1.2]{kmrr2}}\label{KMRR theorem}
For any $A\subset \N$ with positive upper Banach 
density there exist
an infinite set $B\subset A$ 
and some $t\in \N$ such that 
$$B \oplus B := \{b_1 + b_2\colon  b_1,b_2\in B,\ b_1 \neq b_2\} \subset 
A-t.$$
\end{theorem}
Among other results, the first and third authors generalized 
Theorem \ref{KMRR theorem} in the context of abelian groups for 
which the \emph{subgroup of even elements} $2G:= \{2g\colon g\in G\}$, where $2g:=g+g$, has finite index (see \cite{charamaras_mountakis2024}). In \cite{ackelsberg2024counterexamples}, Ackelsberg proved that the finite index assumption is necessary, otherwise there exists a set with upper Banach density arbitrarily close to $1$ that does not contain any shift of $B \oplus B$ for an infinite set $B\subset G$.

A natural question is whether the restriction $b_1\neq b_2$ in Theorem \ref{KMRR theorem} can be removed, thereby allowing the infinite unrestricted sumset $B+B:=\{b_1+b_2\colon b_1,b_2\in B\}$. In contrast to $B\oplus B$, the pattern $B+B$ is not partition regular in $\N$, as shown in \cite{hindman1979partitions} (see also the discussion in \cite[Section 6]{kousek_radic2024}). This suggests that, if such configurations are to be guaranteed inside any dense set $A\subset\N$, one needs a stronger density assumption than the one in Theorem \ref{KMRR theorem}. Indeed, as shown in \cite{kousek_radic2024} by the second and fourth authors, if $A\subset\N$ has natural upper density $\overline{\diff}(A)>2/3$, then there are an infinite set $B\subset\N$ and $t\in\N$ such that $B+B\subset A-t$, and this bound is sharp. The choice of the natural upper density, that is the upper density with respect to the initial \Folner{} sequence\footnote{See \eqref{def_foln} for the definition of \Folner{} sequences.} $N\mapsto[1,N]$, is critical, as there is a set $A\subset\N$ of full upper Banach density that does not contain such a pattern (see \cite[Example 2.3]{kmrr_survey}).
For additional results and open problems on infinite sumsets, we refer
the reader to 
\cite{ackelsberg_nilpotentpolishgroups,host19,kousek2025asymmetric,kmrr1,kmrr_survey,mrr19}.

The main objective of this article is to study the infinite pattern $B+B$ in abelian groups, thus extending the results in \cite{kousek_radic2024}. In this setting, two elementary group parameters, the index $[G:2G]$ and the cardinality of $\ker(D)$ of
the \emph{doubling map} $D\colon G \to G$, $g \mapsto 2g$, play a special role for this pattern. We remark that these parameters are naturally involved in this problem; for instance, when $[G:2G]$ is finite, by the pigeonhole principle, one can reduce $B$ to be in a coset of $2G$, hence $B+B \subset 2G$. A similar phenomenon occurs when $\ker(D)$ is finite.

Therefore, the main question to address is whether for a given abelian group $G$ one can find a \Folner{} sequence $\Phi$ and a constant $c=c(G,\Phi)>0$ such that any set $A\subset G$ with upper density 
$$\overline{\diff}_\Phi(A) := \limsup_{N \to \infty} \frac{|A \cap \Phi_N|}{|\Phi_N|}>c$$ 
contains an unrestricted sumset of the form $B+B$ for some infinite set $B\subset G$. We identify the special feature of the initial \Folner{} sequence that allows one to answer this question in $\N$, and we generalize it to all abelian groups. In what follows, for a set $C \subset G$ we write $2C := D(C)$ and $C/2 := D^{-1}(C)$.

\begin{definition}\label{quasi-invariant to doubling def}
Let $G$ be a countable abelian group. We define {\em the doubling ratio} of a \Folner{} sequence $\Phi=(\Phi_N)_{N\in\N}$ in $G$ as
    \begin{equation} \label{eq alpha_Phi defn}
        \alpha_{\Phi}  = 
        \liminf_{N\to \infty} \frac{| \Phi_N/2 \cap \Phi_N|}{|\Phi_N|}.
    \end{equation}
    Whenever $\alpha_\Phi>0$, we say that the \Folner{} sequence $\Phi$ is 
    {\em quasi-invariant with respect to doubling with ratio $\alpha_\Phi$}, and we 
     abbreviate this as \textit{q.i.d. with ratio $\alpha_{\Phi}$}.
     We also define the {\em group doubling ratio} as
    $\alpha_G = \sup \{ \alpha_{\Phi} \colon \Phi\in\mathcal{F}_G \}$, where $\mathcal{F}_G$ denotes the set of all \Folner{} sequences in $G$. 
\end{definition}

Although the behavior of \Folner{} sequences under endomorphisms, such as the doubling map $D$, has been explored before, formally quantifying their stability under $D$ in this way appears to be new.
In Section \ref{section_existence}, we show that any abelian group $G$ with 
$\ell:=[G:2G]<\infty$ and $r:=|\ker(D)|<\infty$ admits \Folner{} sequences 
that are quasi-invariant with respect to doubling. Furthermore, we prove 
that $\alpha_G=\min\{1,\frac{r}{\ell}\}$. Before stating the main theorem of this article, we shall give some illustrative examples of \Folner{} sequences in $(\Z,+)$ with $\alpha_\Phi$ taking various values within the range $[0,\frac{r}{\ell}]=[0,\frac{1}{2}]$.
\begin{itemize}
    \item Let $\Phi_N = [1,N]$. Then $\Phi_N/2 = [1,\lfloor\frac{N}{2}\rfloor]$, thus $\alpha_\Phi = \frac{1}{2}$. This is a q.i.d. \Folner{} sequence that achieves the maximal doubling ratio.
    \item Let $\Phi_N = [N,2N]$. Then $\Phi_N/2 = [\lceil\frac{N}{2}\rceil,N]$, thus $\alpha_\Phi = 0$. This is not a q.i.d. \Folner{} sequence.
    \item Let $\Phi_N = [N,4N]$. Then $\Phi_N/2 = [\lceil\frac{N}{2}\rceil,2N]$, thus, $\alpha_\Phi=\frac{1}{3}$.
\end{itemize}

Given any value $\alpha$ in the range $(0,\frac{1}{2}]$, one can 
modify the last example appropriately to construct a q.i.d. \Folner{} sequence in $(\Z,+)$ whose doubling ratio is $\alpha$.
We are now ready to state the main theorem.

\begin{theorem} \label{main_theorem_1}
Let $(G,+)$ be a countable abelian group with both $\ell=[G:2G]$ and $r=|\ker(D)|$ finite. 
Let also $A\subset G$ and $\Phi$ be any \Folner{} sequence in $G$ that is quasi-invariant with respect to doubling with ratio $\alpha_{\Phi}$. Then the following hold:
\begin{enumerate}
\item  \label{main_theorem_1_2}
If 
$\displaystyle \overline{\diff}_\Phi(A) > 1 - \frac{\ell\alpha_\Phi}{\ell +r},$
then there exists an infinite set $B\subset G$ and some $t\in G$ such that 
$t+B+B\subset A$. 
\item  \label{main_theorem_1_1} 
\vspace{2mm} 
If  
$\displaystyle { \overline{\diff}_\Phi(A) > 1 - \frac{\alpha_\Phi}{\ell+r}},$
then there exists an infinite set $B\subset G$ such that $B+B\subset A$. 
\item   \label{main_theorem_evens}  \vspace{2mm} If 
$\displaystyle\overline{\diff}_\Phi(A\cap 2G) > \frac{1}{\ell} - \frac{\alpha_\Phi}{\ell +r},$
then there exists an infinite set $B\subset G$ such that $B+B\subset A$.
\end{enumerate}
\end{theorem}

The statements (1), (2) and (3) of Theorem \ref{main_theorem_1} are logically equivalent (see Section \ref{section_equivalence}), hence it suffices to establish only one of them. In Section \ref{section_dynamical_st}, we translate (2) into a dynamical statement which is proved in Section \ref{section_proof_of_dynamical_st}, concluding the proof of Theorem \ref{main_theorem_1}.

The second main objective of this paper is to show sharpness of the lower bounds in Theorem \ref{main_theorem_1}. This is not evident from the proof: the bounds are directly derived from a variant of Furstenberg's correspondence principle given in Lemma \ref{correspondence}. The next theorem addresses this question, showing optimality with respect to the parameters $\ell$, $r$, and $\alpha_G=\min\{1,\frac{r}{\ell}\}$. This indicates that the constant $\alpha_\Phi$ is not an artificial device of the proof, but rather captures an existing phenomenon. We highlight that in all abelian groups $G$, both $[G:2G]$ and $|\ker(D)|$, whenever finite, are powers of $2$ (see Lemma \ref{powers_of_2}).

\begin{theorem} \label{optimality}
Let $\ell,r \in \N$ be powers of $2$. Then, there exist a countable 
abelian group 
$G$ with $[G:2G]=\ell$ and $|\ker(D)|=r$, a \Folner{} sequence $\Phi$ 
in $G$ which is quasi-invariant with respect to doubling with ratio
$\alpha_{\Phi}=\alpha_G$, and a set $A\subset G$ with 
$\overline{\diff}_{\Phi}(A)=1-\frac{\alpha_G}{\ell +r}$, 
for which $B+B \not \subset A$ for any infinite $B\subset G$.
\end{theorem}

This result was already known in $\Z$
(see \cite[Section 4]{kousek_radic2024}). The proof of Theorem \ref{optimality} is carried out in Section \ref{section_examples} where for every choice of $\ell, r$, we construct examples
using the groups $\Z$,  $\F_p^{\omega}$, 
$\Z(1/2)/\Z = \{ \frac{k}{2^n} \mod 1 \mid n, k \in \Z\}$ and their products. The most technical and intricate task in the proof is the case $\ell = 2^{d_1}$ and $r = 2^{d_2}$, with 
$d_1, d_2\geq 1$.
The following table summarizes the groups where we build the corresponding examples with their respective value of $\alpha_G$. 

\begin{table}[H]
\centering
\begin{tabular}{|c|c|c|}
\hline
& $\ell =1$ & $\ell = 2^{d_1}, d_1\geq1$ \\ \hline
$r=1$ & $\F_3^{\omega}, \quad \alpha_G = 1$ & $\Z^{d_1}, \quad \alpha_G = 2^{-d_1}$ \\ \hline
$r=2^{d_2}, d_2\geq1$ & $(\Z(\frac{1}{2})/\Z)^{d_2}, \quad \alpha_G = 1$ & $\Z^{d_1} \times (\Z(\frac{1}{2})/\Z)^{d_2}, \quad \alpha_G = \min\{1,2^{d_2-d_1}\}$ \\ \hline
\end{tabular}
\end{table}

In a similar spirit, we show that the assumptions concerning the group $G$ and the \Folner{} sequence in Theorem \ref{main_theorem_1} are necessary (see Section \ref{section_necessity}). In particular, we construct an abelian group with infinite $\ker(D)$, and a set of full density along a \Folner{} sequence that is quasi-invariant with respect to doubling, which contains no infinite sumsets $B+B$ (see Section \ref{finite_kernel_is_necessary}).
In addition, in Section \ref{section qid is necesary} we show that 
along any \Folner{} sequence
that is not quasi-invariant with respect to doubling, 
there are sets of full upper density 
that contain no
infinite sumsets $B+B$. Finally, recall that the assumption of the finiteness of $[G:2G]$ is imposed by \cite{ackelsberg_nilpotentpolishgroups}. However, the results there do not strictly rule out the possibility that there is an abelian group $G$ with $[G:2G]=\infty$ such that any set of full upper Banach density contains an infinite sumset $t+B+B$. In Section \ref{inf_index}, we show that this is the case only if $2G$ is finite.

We end this section with open questions arising naturally from our work.
The first is related to recent results \cite{ackelsberg_nilpotentpolishgroups} for abelian groups $G$ with finite $[G\colon 6 G]$ and sets that contain a pattern of the form $B \oplus B \oplus B$.

\begin{question}
Is it possible to characterize the abelian groups and the \Folner{} sequences for which there is a constant $c >0$ such that if $A \subset G$ and $\overline{\diff}_\Phi (A) > c$, then there exists $B \subset G$ infinite such that $B+B+B \subset A$?
\end{question}

This problem is significantly different from the 
one we consider. Indeed, while $B+B = 2B \cup (B \oplus B)$, $B+B+B$ is not equal to $3B \cup  (B \oplus B \oplus B)$ as one must additionally account for the sets $\{ b_1 + 2b_2 \colon b_1, b_2 \in B, b_1 < b_2 \}$, $\{ 2b_1 + b_2 \colon b_1, b_2 \in B, b_1 < b_2 \}$. Those two extra components change the nature of the combinatorial problem; loosely speaking, it is like handling four linear equations instead of two.

The next two questions deal with stronger optimality results for the bounds in Theorem \ref{main_theorem_1}.

\begin{question}\label{open_qu_0}
Let $G$ be a countable abelian group with $\ell=[G:2G]<\infty$ and 
$r=|\ker(D)|<\infty$, and let $\Phi$ be a \Folner{} sequence in $G$ that is quasi-invariant with respect to doubling. Does there exist a set $A\subset G$ with 
$\overline{\diff}_{\Phi}(A)=1-\frac{\alpha_{\Phi}}{\ell +r}$, 
such that $B+B \not \subset A$ for any infinite $B\subset G$?
\end{question}

Question \ref{open_qu_0} asks for the strongest possible 
notion of optimality for the bounds in our main result. The
following, weaker, natural question has better chances of 
having a positive answer.

\begin{question}\label{open_qu_1}
Let $G$ be a countable abelian group with $\ell=[G:2G]<\infty$ and 
$r=|\ker(D)|<\infty$. Can one always find a \Folner{} sequence $\Phi$ in $G$ 
that is quasi-invariant with respect to doubling and a set 
$A\subset G$ with 
$\overline{\diff}_{\Phi}(A)=1-\frac{\alpha_{\Phi}}{\ell +r}$, 
such that $B+B \not \subset A$ for any infinite $B\subset G$?
\end{question}

We stress that Theorem \ref{optimality} does not provide an answer to
Question \ref{open_qu_1}. Indeed, given the parameters $\ell$ and $r$, Theorem \ref{optimality} asserts the existence of some group $G$ with those values, a set $A\subset G$ and a \Folner{} sequence $\Phi$ in $G$ such that $\overline{\diff}_{\Phi}(A)=1-\frac{\alpha_G}{\ell +r}$ and
$B+B \not \subset A$ for any infinite $B\subset G$.

\vspace{2mm} 
\noindent
\textbf{Notational conventions.} 
We let $\N=\{1, 2, \ldots\}$ and $\N_0=\{0, 1, 2, \ldots\}$. 
In addition, we use the symbol $\sqcup$ to denote unions of pairwise disjoint
sets. Given an abelian group $G$, we denote the identity element by $0$, and
$\displaystyle G^{\omega}$ is the group of 
eventually zero sequences in $G$.
Finally, we write $\oh_{k \to \infty}(1)$ to 
denote an error term that goes to $0$ as $k$ grows to infinity.

\vspace{2mm} 
\noindent
\textbf{Acknowledgements.} 
The authors would like to thank Bryna Kra, Nikos Frantzikinakis,
Joel Moreira and Florian K. Richter for providing helpful comments and 
suggestions, and the anonymous referees for the valuable feedback that helped improve this article. This work was completed while the first author was
in École Polytechnique Fédérale de Lausanne, and was supported by the Swiss National Science Foundation
grant TMSGI2-211214. 
The second author was supported by the Warwick Mathematics Institute 
Centre for Doctoral Training. This work was completed while the third author was
in the University of Crete, and was supported by the 
Research Grant ELIDEK HFRI-NextGenerationEU-15689. The fourth author was partially supported by the National Science Foundation grant DMS-2348315.

\section{Translation of Theorem \ref{main_theorem_1} to a dynamical statement}
\label{section_2}

The proof of Theorem \ref{main_theorem_1} is accomplished via a
dynamical systems reformulation.  
We start the present section by showing that any one of the statements \eqref{main_theorem_1_2}, \eqref{main_theorem_1_1} or 
\eqref{main_theorem_evens} of Theorem \ref{main_theorem_1} implies the others,
and hence  
it suffices to prove \eqref{main_theorem_1_1} in order 
to 
establish the theorem. In Section \ref{subsec_background} we
provide the background material needed in order to realize
the proof. Finally, in
Section \ref{section_dynamical_st} we state our main dynamical
result, 
Theorem \ref{analogue of 2.1 KR}, and prove 
that it implies Theorem \ref{main_theorem_1}.

\subsection{Equivalence of the statements in the main theorem}\label{section_equivalence}
 
 As was mentioned before, $G$ always denotes a countable abelian group with $\ell=[G: 2G]< \infty$. We also fix $g_1, \ldots, g_{\ell}\in G$ so that $G=\bigsqcup_{i=1} ^{\ell} 2G+g_i$. We omit these assumptions from the statements of this section. 

\begin{lemma} \label{lemma reductions no shift}
    Fix a subset $A \subset G$. The following are equivalent:
    \begin{enumerate}
        \item \label{lemma reductions no shift_1} $A$ contains $B+B$ for some 
        infinite set $B\subset G$.
        \item \label{lemma reductions no shift_2} $A \cap 2G$ contains $B+B$ for some infinite set $B\subset G$.
        \item \label{lemma reductions no shift_3} $(A \cap 2G) \cup(G \backslash 2G) $ contains $B+B$ for some infinite set $B\subset G$.
    \end{enumerate}
\end{lemma}

\begin{proof}
    For any infinite $B\subset G$, by the pigeonhole principle, there is an 
    infinite subset of $B$ which is contained in some coset $2G+g_i$. Equivalently, there exists an infinite $B' \subset 2G$ and some $i\in\{1,\dots,\ell\}$, such that 
    $B'+g_i \subset B$. If $B+B\subset A$, then for $B'+g_i$
    chosen as before we have $(B'+g_i)+(B'+g_i)\subset A\cap 2G$,
    therefore proving \eqref{lemma reductions no shift_1} $\implies$
    \eqref{lemma reductions no shift_2}. The implications \eqref{lemma reductions no shift_2} $\implies$ \eqref{lemma reductions no shift_1} and 
    \eqref{lemma reductions no shift_2} $\implies$
    \eqref{lemma reductions no shift_3} are obvious. Finally, that \eqref{lemma reductions no shift_3} implies \eqref{lemma reductions no shift_2} is a special case of the implication 
    \eqref{lemma reductions no shift_1} $\implies$
    \eqref{lemma reductions no shift_2}, because $\left( (A\cap 2G) \cup (G \setminus 2G) \right) \cap 2G = A\cap 2G$.
\end{proof}

\begin{lemma} \label{lemma reductions with shift}
Fix a subset $A \subset G$. The following are equivalent:
\begin{enumerate}
    \item \label{lemma reductions with shift_1} $A$ contains $t+B+B$ for some $t \in G$ and infinite $B \subset G$.
    \item \label{lemma reductions with shift_2} $A$ contains $B+B + g_i$ for some $i \in \{1, \ldots, \ell\}$ and infinite 
    $B \subset G$.
    \item \label{lemma reductions with shift_3} $(A - g_i) \cap 2G $ contains $B+B$ for some $i \in \{1, \ldots, \ell\}$ and infinite 
    $B \subset G$.
\end{enumerate}
\end{lemma}

\begin{proof}
    The implication \eqref{lemma reductions with shift_1}$\implies$\eqref{lemma reductions with shift_2}
    uses the fact that any $t\in G$ can be written as $2s + g_i$ for some 
    $s \in G$ and $i \in \{1, \ldots,\ell\}$. Thus if we define $B'= B+s$ we 
    get that $B'+B'+g_i = B+B+2s+ g_i \subset A$. Assuming \eqref{lemma reductions with shift_2}, we have that $B+B \subset A-g_i$ and then \eqref{lemma reductions with shift_3} follows directly from the equivalence of \eqref{lemma reductions no shift_1} and \eqref{lemma reductions no shift_2} in Lemma \ref{lemma reductions 
    no shift}. Finally,  
    \eqref{lemma reductions with shift_3}$\implies$\eqref{lemma reductions with shift_1} is obvious. 
\end{proof}

Using the previous lemmas we deduce the following proposition.
\begin{proposition} \label{prop equiv formulation}
    Let $ \Phi$ be a F\o lner sequence in $G$, $A \subset G$ and $\beta >0$. Then the following statements are equivalent:
    \begin{enumerate}
        \item \label{prop equiv formulation_2} If $\overline{\diff}_{ \Phi}(A ) > \ell \beta$, then $A$ contains $t+B+B$ for some $t \in G$ and some infinite set $B\subset G$.
        \item \label{prop equiv formulation_3} If $\overline{\diff}_{ \Phi}(A ) > \beta + \frac{\ell - 1}{\ell}$ then $A$ contains 
        $B+B$ for some infinite set $B\subset G$.
        \item \label{prop equiv formulation_1} If $\overline{\diff}_{ \Phi}(A \cap 2G) > \beta$ then $A$ contains $B+B$ 
        for some infinite set $B\subset G$.
    \end{enumerate} 
\end{proposition}

\begin{proof}
    \eqref{prop equiv formulation_2}$\implies$\eqref{prop equiv formulation_1}: If $\overline{\diff}_{ \Phi}(A \cap 2G ) > \beta$, then we define $\Tilde{A} = \bigsqcup_{i=1}^{\ell} (A \cap 2G) + g_i$, and we have $\overline{\diff}_{ \Phi}(\Tilde{A} ) > \ell \beta$. By \eqref{prop equiv formulation_2}, $\Tilde{A}$ contains a $t+B+B$ for some infinite $ B \subset G$ and $t \in G$.
    By Lemma \ref{lemma reductions with shift} we can reduce to the case where $t= g_i$ for some $i \in \{1, \ldots, \ell\}$. Using 
    Lemma \ref{lemma reductions no shift}, we can also assume that $B+B \subset 2G$. Therefore, we have that $B+B+ g_i \subset (A \cap 2G) + g_i$ which concludes the proof.

    \eqref{prop equiv formulation_1}$\implies$\eqref{prop equiv formulation_2}: If $\overline{\diff}_{ \Phi}(A ) > \ell \beta$, then, by sub-additivity of the density, $\overline{\diff}_{ \Phi}(A \cap (2G + g_i) ) > \beta$ for some $i \in \{1, \ldots, \ell\}$. By translation invariance
    of the density, we see that $\overline{\diff}_{\Phi}((A-g_i) \cap 2G  ) > \beta$ and therefore by \eqref{prop equiv formulation_1}, $A$ contains $B+B+g_i$ for some $ B \subset G$ infinite.
    
    \eqref{prop equiv formulation_3}$\implies$\eqref{prop equiv formulation_1}: If $\overline{\diff}_{ \Phi}(A \cap 2G) > \beta$, then since $\diff_{\Phi}(G\backslash 2G) = \frac{\ell-1}{\ell}$,
    we have $\overline{\diff}_{ \Phi}((A \cap 2G) \cup(G \backslash 2G)) > \beta + \frac{\ell - 1}{\ell}$. Therefore, by $(2)$, $(A \cap 2G) \cup(G \backslash 2G)$ contains $B+B$ and we conclude using Lemma \ref{lemma reductions no shift}.

    \eqref{prop equiv formulation_1}$\implies$\eqref{prop equiv formulation_3}: If $\overline{\diff}_{ \Phi}(A ) > \beta + \frac{\ell - 1}{\ell}$ then, using sub-additivity of the density and 
    the fact that each coset has density $\frac{1}{\ell}$, we get
    $\overline{\diff}_{ \Phi}(A \cap 2G) > \beta$, so we conclude using \eqref{prop equiv formulation_1}.
\end{proof}

From Proposition \ref{prop equiv formulation}, it is now immediate that 
\eqref{main_theorem_1_2}, \eqref{main_theorem_1_1} and 
\eqref{main_theorem_evens} of Theorem \ref{main_theorem_1} are equivalent. 
Therefore, it suffices to prove \eqref{main_theorem_1_1} in order to 
establish the theorem.

\subsection{Terminology and background from ergodic theory}
\label{subsec_background}

Throughout, let $G$ be a countable abelian group. 
Given a compact metric space $X=(X,d_X)$, 
a \textit{continuous action} $T=(T_g)_{g\in G}$ of $G$ on X
is a collection of continuous functions $T_g:X\to X$ such that 
for any $g_1, g_2 \in G$, $T_{g_1} \circ T_{g_2} = T_{g_1+g_2}$.
Given such an action, we call the pair $(X,T)$ a \textit{topological $G$-system.}

Fix a topological $G$-system $(X,T)$. A measure $\mu$ in the space of 
Borel probability measures on $X$ is said to be $T$-invariant, 
if it is invariant under $T_g$ for all $g\in G$. 
The Borel $\sigma$-algebra on $X$ is denoted by $\B_X$ or just $\B$, 
if no confusion may arise. The action $T$ on the Borel probability space 
$(X,\mu)$ is called a \textit{measure-preserving $G$-action} and 
$\xmt$ is called a \textit{measure-preserving $G$-system}. For simplicity, we refer to these as $G$-actions, and $G$-systems, respectively. 
A $G$-system $\xmt$ is called \textit{ergodic} 
if for any measurable set $A$
the following holds:
$$ T_g ^{-1} A =A \text{ for all } g\in G \implies \mu(A)=0 \text{ or }
\mu(A)=1.$$
Given a $G$-system $\xmt$ and a \Folner{} sequence $\Phi$, a point $a\in X$ is called 
\textit{generic with respect to $\mu$ along $\Phi$} 
if for all $f\in C(X)$ we have 
$$ \lim_{N\to \infty} \frac{1}{|\Phi_N|} \sum_{g\in \Phi_N} f(T_g a)
= \int_{X} f \d \mu, $$
or equivalently if 
$$ \lim_{N\to \infty} \frac{1}{| \Phi_N|} \sum_{g\in \Phi_N}
\delta_{T_g a}=\mu, $$
where $\delta_x$ is the Dirac mass at $x\in X$ and the limit 
is in the weak$^\ast$ topology. If $a$ is generic 
for $\mu$ along $\Phi$, then we denote this by 
$a\in \gen(\mu,\Phi)$. Moreover, 
we let $\text{supp}(\mu)$ denote the 
\textit{support} of $\mu$, that is, the 
smallest, closed, full-measure 
(with respect to $\mu$) subset of $X$.

Given two $G$-systems $\xmt$ and $\yns$, we say that $\yns$ is a
\textit{factor} of $\xmt$ if there exists a measurable map 
$\pi:X\to Y$, which we call \textit{factor map}, such that
$\mu(\pi^{-1}E)=\nu(E)$ for any measurable $E\subset Y$, and for any $g\in G$,
$\pi\circ T_g = S_g\circ\pi$ holds $\mu$-almost everywhere on $X$. We say that $\nu$ 
is the pushforward of $\mu$
under $\pi$, and we write $\pi\mu=\nu$. When, additionally, the 
factor map $\pi$ is continuous and $\pi\circ T_g=S_g\circ\pi$ holds 
everywhere on $X$ for any $g\in G$, we say that $\pi$ is a \textit{continuous
factor map} and $\yns$ is a \textit{continuous factor} of $\xmt$. 

An important class of 
factors is those with the 
structure of group rotations. In 
particular, for any ergodic 
$G$-system $\xmt$ we utilize 
its \textit{Kronecker factor}, 
which is the maximal factor of the system that is isomorphic to an abelian 
group rotation (see \cite[Theorem 1]{mackey}): There exists a compact 
abelian group $Z$ and a group homomorphism 
$\theta:G\to Z$ with dense 
image, such that the Kronecker factor of $\xmt$ is measurably isomorphic 
to $(Z,m,R)$, where $m$ is 
the normalized Haar measure on $Z$ and $R$ is the rotation by $\theta$, i.e. 
for all $z\in Z$ and $g\in G$,
\begin{equation} \label{eq kroneker}
R_g(z) = \theta(g) + z.  
\end{equation}

\subsection{Dynamical reformulation of Theorem \ref{main_theorem_1}}
\label{section_dynamical_st}
To prove Theorem \ref{main_theorem_1}, 
we follow an ergodic theoretic approach. In particular, we translate 
the problem of finding infinite sumset configurations in subsets of a group 
$G$ to a statement in ergodic theory, and particularly
to the existence of Erd\H{o}s progressions in 
products of $G$-systems. These ideas and 
basic tools necessary for their implementation were 
developed in \cite{kmrr1} and 
\cite{kmrr2} in the context of finding infinite 
configurations in $\N$, and were 
subsequently exploited for finding other patterns 
in $\N$ (see \cite{kousek_radic2024}) and 
generalized in the context of countable amenable 
groups (see \cite{charamaras_mountakis2024}).

Again, let $(G,+)$ be an abelian group with $\ell=[G: 2G] < \infty$ and $r=|\ker(D)|<\infty$. For the rest of 
this section $\Sigma$ denotes the space $\{0,1\}^{G}$, and it is endowed 
with 
the product topology, so that it becomes a compact metrizable space. 
We also consider the shift action $S=(S_g)_{g\in G}$ defined by
$S_g\colon \Sigma \to \Sigma$, $S_g(x(h))=x(h+g)$, for any 
$h,g\in G$, $x=(x(g))_{g\in G} \in \Sigma$, and note that 
$S$ is an action of $G$ on $X$ by homeomorphisms.
We further denote the $G$-action $g\mapsto S_{2g}$ by $S^2$.

We use the following variant of Furstenberg's correspondence principle, (originally introduced in \cite{furstenberg_szemeredi_proof}), to reduce 
Theorem \ref{main_theorem_1} to a dynamical statement.
This variant crucially exploits the special 
structure of q.i.d. \Folner{} sequences. A similar version appears in \cite[Lemma 2.7]{kousek_radic2024}.

\begin{lemma}\label{correspondence}
Let $A\subset G$ and $\Phi$ be any \Folner{} sequence in $G$ that is q.i.d. with ratio 
$\alpha_{\Phi}$. Then there exist a $\beta\geq \alpha_{\Phi}$,
an ergodic $G$-system $(\Sigma 
\times \Sigma, \mu, S^2 \times S)$, an open set $E\subset \Sigma$, a point $a\in \Sigma$ and a \Folner{} sequence 
$\Phi'$, such that $(a,a)\in \gen(\mu,\Phi')$, $A=\{g\in G\colon S_g a \in E\}$ and 
\begin{equation}\label{dbound_no_shift}
    \ell\mu(\Sigma\times E) + \mu(E\times\Sigma) 
    \geq
    \frac{\ell+r}{\beta}\left(\overline{\diff}_\Phi(A)-1\right) + \ell + 1.
\end{equation}
\end{lemma}

\begin{proof}

By definition, and passing to a subsequence of $(\Phi_N)$ if necessary, we may assume that there exists some $\beta \geq \alpha_{\Phi}$ so that
$$\lim_{N\to\infty}\frac{|\Phi_N/2 \cap \Phi_N|}{|\Phi_{N}|}=\beta~~ \text{ and } ~~\overline{\diff}_\Phi(A) = \lim_{N\to\infty}\frac{|A\cap\Phi_{N}|}{|\Phi_{N}|}.$$
Associate to the set $A$ a point $a\in\Sigma=\{0,1\}^G$ via 
$$a(g) = \begin{cases}
    1, & \text{if}~g\in A, \\
    0, & \text{otherwise}.
\end{cases}$$

Define the clopen set $E=\{x\in\Sigma\colon x(0)=1\}$ and observe that, by 
construction, $A=\{g\in G\colon S_ga\in E\}$. Since $\Phi$ is quasi-invariant with respect
to doubling, by Lemma \ref{lemma_phi/2_phi} we have that $\Psi=(\Psi_N)_{N\in\N}$ 
given by $\Psi_N = \Phi_N/2\cap\Phi_N$ is a \Folner{} sequence, and then we 
can define the sequence of Borel probability measures $(\mu_N)_{N\in \N}$ on 
$\Sigma\times\Sigma$ given by 
$$\mu_N = \frac{1}{|\Psi_{N}|}\sum_{g\in \Psi_{N}}\delta_{(S_{2g}\times S_g)(a,a)}.$$
We let $\mu'$ be a weak* accumulation point of $(\mu_N)_{N\in \N}$, and  
then it is easy to see that $\mu'$ is an $S^2 \times S$-invariant measure.
Recall that
$\lim_{N\to\infty}\frac{|\Psi_{N}|}{|\Phi_{N}|}=\beta .$

For each $N$, $\Psi_{N}\subset\Phi_{N}$, so we have
$$\mu_N(\Sigma\times E) = \frac{1}{|\Psi_{N}|}\sum_{g\in \Psi_{N}}\delta_{S_ga}(E)
\geq \frac{|\Phi_{N}|}{|\Psi_{N}|}\frac{1}{|\Phi_{N}|}(|A\cap\Phi_{N}| - |\Phi_{N}| + |\Psi_{N}|),$$
and then sending $N\to\infty$ yields
\begin{equation}\label{dbound_eq1}
    \mu'(\Sigma\times E) \geq \frac{1}{\beta}\left(\overline{\diff}_\Phi(A)-1+\beta \right)
    = \frac{\overline{\diff}_\Phi(A)}{\beta} - \frac{1}{\beta} + 1.
\end{equation}
From Lemma \ref{lemma aux folner 3} \eqref{item_u_f} and the finiteness of $\ker(D)$ we know that $N\mapsto F_{N}=\bigcup_{g\in \Psi_{N}}
g+\ker(D) \supset \Psi_{N}$ is a \Folner{} sequence in $G$, and from Lemma \ref{lemma aux folner} we deduce that
$\frac{|F_{N}|}{|\Psi_{N}|}\to 1$ as $N\to \infty$. From the definition of 
$F_{N}$ we get $\sum_{g\in F_{N}} \delta _{S_{2g}a} (E)= r 
\sum_{g\in 2 \Psi_{N}} \delta _{S_{g}a} (E)$, and therefore
\begin{equation*}
    \left| \frac{1}{|\Psi_{N}|} \sum_{g\in \Psi_{N}} \delta _{S_{2g}a} (E) 
    - \frac{r}{|\Psi_{N}|} \sum_{g\in 2 \Psi_{N}} \delta _{S_{g}a} (E)
    \right| \leq 
    \frac{|F_{N} \setminus \Psi_{N}|}{|\Psi_{N}|}
\end{equation*}
which goes to $0$ as $N\to \infty$.
In addition, for each $N\in\N$, $2\Psi_{N} = 2(\Phi_{N}/2\cap\Phi_{N}) 
\subset \Phi_{N}\cap 2\Phi_{N} \subset \Phi_{N}$, so combining with the 
previous we get
\begin{align*}
    \mu_N(E\times\Sigma) = \frac{1}{|\Psi_{N}|}
    \sum_{g\in\Psi_{N}}\delta_{S_{2g}a}(E)
    & = \frac{r}{|\Psi_{N}|}\sum_{g\in 2\Psi_{N}}\delta_{S_ga}(E) + \oh_{N\to\infty}(1) \notag \\
    & \geq r\frac{|\Phi_{N}|}{|\Psi_{N}|}\frac{1}{|\Phi_{N}|}(|A\cap\Phi_{N}| - |\Phi_{N}| + |2\Psi_{N}|) + \oh_{N\to\infty}(1),
\end{align*}
and then sending $N\to\infty$ yields
\begin{equation}\label{dbound_eq2}
    \mu'(E\times\Sigma) 
    \geq \frac{r}{\beta}\bigg(\overline{\diff}_\Phi(A) - 1 + 
    \frac{\beta}{r}\bigg)
    = \frac{r\overline{\diff}_\Phi(A)}{\beta} - \frac{r}{\beta} + 1,
\end{equation}
where we used that by Lemma \ref{lemma aux folner}, $\frac{|2\Psi_{N}|}{|\Psi_{N}|} \to \frac{1}{r}$, so $\frac{|2\Psi_{N}|}{|\Phi_{N}|} \to
\frac{\beta}{r}$ as $k\to \infty$.
Combining \eqref{dbound_eq1} and \eqref{dbound_eq2} 
we obtain \eqref{dbound_no_shift} for $\mu'$. 
Although $\mu'$ is not necessarily ergodic, we 
can use its ergodic 
decomposition to find an $(S^2 \times S)$-ergodic 
component of it, call it $\mu$, so that \eqref{dbound_no_shift} holds for 
$\mu$ as well. Without loss of generality we may 
assume that 
$\mu$ is supported on the orbit closure of $(a,a)$, 
since this holds for $\mu'$ by construction. Then by a standard argument (see \cite[ Proposition 3.9]{Fur1}) we see there is a \Folner{} sequence $\Phi'$
in $G$ such that $(a,a) \in \gen(\mu,\Phi')$. 
\end{proof}

We now state the main dynamical result, which, along with Lemma
\ref{correspondence} and Lemma \ref{EP and B+B} below,
allows us to prove Theorem \ref{main_theorem_1}.
For that, we need the notion of Erd\H{o}s progressions, introduced for $\Z$-actions in \cite{kmrr2}.
Given a topological $G$-system $(X,T)$, a triple 
$(x_0,x_1,x_2) \in X^3$ is called a 
($3$-term) Erd\H{o}s progression if there exists 
an infinite sequence $(g_n)_{n\in \N}$ in $G$ (that is, the set 
$\{g_n : n\in \N\}$ is infinite)
such that $(T_{g_n} \times T_{g_n})(x_0,x_1) \xrightarrow{} (x_1,x_2)$ 
as $n\to \infty$.

\begin{theorem} \label{analogue of 2.1 KR}
Let $(X,\mu,T)$ be an ergodic $G$-system, $a\in \gen(\mu,\Phi)$ for some \Folner{} sequence $\Phi$ in $G$,  
and $E_1,E_2 \subset X$ be open sets satisfying
\begin{equation}\label{1_1}
\ell \mu(E_2)+\mu(E_1) > \ell.
\end{equation}
Then, there exists an \Erdos{} progression $(a,x_1,x_2)$ such that $(x_1,x_2) \in E_1 \times E_2$.
\end{theorem}

We note here that Theorem \ref{analogue of 2.1 KR} was already suggested
by Tao (see the discussion after Theorem 7 in 
\cite{Tao_sumsets}).
We postpone the proof of Theorem \ref{analogue of 2.1 KR} to 
Section \ref{dynamical results}.

\begin{lemma}\cite[Lemma $3.4$]{charamaras_mountakis2024} \label{EP and B+B} 
Let $(X,T)$ be a topological $G$-system and let $E,F\subset X$ be open. Assume 
there exists an Erd\H{o}s progression $(x_0,x_1,x_2) \in X^3$ with $x_1\in E$ 
and $x_2\in F$. Then, there exists an infinite sequence $B=(b_n)_{n\in \N} 
\subset \{g\in G: T_g(x_0) \in E\}$ such that $B \oplus B =
\{b_n + b_m \colon n,m \in \N, n \neq m\} \subset \{g\in G: 
T_g(x_0)\in F\}$.
\end{lemma}

As Lemma \ref{EP and B+B} suggests, the existence  
of the \Erdos{} progressions in 
the context of Theorem \ref{analogue of 2.1 KR} is what allows us 
to recover the combinatorial statements of Theorem \ref{main_theorem_1}. 

\begin{proof}[Proof that Theorem \ref{analogue of 2.1 KR} implies Theorem \ref{main_theorem_1}]
Let $G, \ell, r, \Phi$ and $\alpha_{\Phi}$ be as in the assumptions of 
Theorem \ref{main_theorem_1}. Let also $A\subset G$ with 
$\overline{\diff}_\Phi(A) > 1 - \frac{\alpha_\Phi}{\ell+r}.$
Using Lemma \ref{correspondence} we can then find $\beta\geq \alpha_{\Phi}$,
an ergodic $G$-system $(\Sigma 
\times \Sigma, \mu, S^2 \times S)$, an open set $E\subset \Sigma$, a point $a\in \Sigma$ and a \Folner{} sequence 
$\Phi'$ in $G$, such that $(a,a)\in \gen(\mu,\Phi')$, $A=\{g\in G\colon S_g a \in E\}$ and 
\begin{equation}
    \ell\mu(\Sigma\times E) + \mu(E\times\Sigma) 
    \geq
    \frac{\ell+r}{\beta}\left(\overline{\diff}_\Phi(A)-1\right) + \ell + 1.
\end{equation}
From the assumption on $\overline{\diff}_{\Phi}(A)$ we get that 
$\ell\mu(\Sigma\times E) + \mu(E\times\Sigma) >\ell$, so using Theorem \ref{analogue of 2.1 KR} for the system $(\Sigma \times \Sigma, \mu, S^2 \times S)$, the open sets $\Sigma \times E, E \times \Sigma$ and the point $(a,a)\in 
\gen(\mu, \Phi ')$ we get 
that there is an \Erdos{} progression $((a,a), (x_{11}, x_{12}), (x_{21}, x_{22})) \in \Sigma ^6 $ with 
$$(x_{11}, x_{12}, x_{21}, x_{22})\in E \times \Sigma \times \Sigma \times E.$$
We can now apply Lemma \ref{EP and B+B} for the sets $U=E\times \Sigma$ and 
$V=\Sigma \times E$ to get an infinite sequence $B=(b_n)_{n\in \N}$
so that 
\begin{equation}\label{eq_useful_2B}
    B \subset \{g\in G: S_{2g}\times S_{g} (a,a) \in E\times \Sigma\}
    =\{ g\in G: S_{2g} a \in E \}
\end{equation}
and 
\begin{equation}\label{eq_useful_B+B}
    B \oplus B \subset \{g\in G: S_{2g}\times S_{g} (a,a) \in \Sigma \times E\}
    =\{g\in G: S_{g} a \in E\}=A.
\end{equation}
Let us denote by $2B$ the set $2B=\{2b: b\in B\}$.
From \eqref{eq_useful_2B} we get that 
$2B\subset A$, so combining with \eqref{eq_useful_B+B} and the fact that $B+B=
(B \oplus B) \cup 2B$ we get that $B+B\subset A$. This establishes part 
\eqref{main_theorem_1_1} of Theorem \ref{main_theorem_1}, and since by Proposition \ref{prop equiv formulation} 
the three statements of this theorem are equivalent, it concludes the proof.
\end{proof}

\section{Measures on \Erdos{} progressions and the proof of the dynamical statement}
\label{section_proof_of_dynamical_st}

In this section we prove Theorem \ref{analogue of 2.1 KR}. In 
Section \ref{properties of sigma} we collect some useful tools, and then 
in Section \ref{dynamical results} we present the proof of the theorem.

\subsection{Erd\H{o}s progressions, measures and their 
properties} \label{properties of sigma}

Again, let us fix a countable abelian group $(G,+)$ with $\ell=[G:2G]<\infty$, and
${g}_1=e_{G},g_2, \ldots, {g}_{\ell}\in G$ such that
$G = \bigsqcup_{i=1}^{\ell} ({g}_i + 2G)$.
Moreover, we fix an ergodic $G$-system $\xmt$ admitting a 
continuous factor map $\pi: X\to Z$ to its Kronecker factor $\zmr$,
a \Folner{} sequence $\Phi$ in $G$ and a point $a\in \textbf{gen}
(\mu,\Phi)$.

In Theorem \ref{analogue of 2.1 KR} we care about Erd\H{o}s progressions
with first coordinate equal to $a$. 
To this end, we utilize a natural measure $\sigma_a$ on $X\times X$ with the property that $\sigma_a$-almost every pair $(x_1,x_2)$ is such 
that $(a,x_1,x_2)$ projects under the factor map 
$\pi$ to an Erd\H{o}s progression $(\pi(a),\pi(x_1),\pi(x_2))$ in the 
Kronecker factor $Z$. Although the definition of this 
measure is not necessary here, we include it for 
completeness. 

Let $z\mapsto \eta_z$ denote the disintegration of 
$\mu$ over the factor map $\pi$ (for details see
\cite[Section 
5.3]{Eiseidler&Ward:2011}) and for every $ (x_1,x_2) \in X \times X$, 
consider the measure 
\begin{equation} \label{lambda}
\lambda_{(x_1,x_2)}=\int_Z \eta_{z+\pi(x_1)} \times \eta_{z+\pi(x_2)}\ dm(z)
\end{equation} 
on $X \times X$. Then
$(x_1,x_2) \mapsto \lambda_{(x_1,x_2)}$ is a \emph{continuous ergodic 
decomposition} of $\mu \times \mu$, that is, a disintegration of $\mu \times 
\mu $ where the measures $\lambda_{(x_1,x_2)}$ are ergodic for $\left( \mu 
\times \mu \right)$-almost every $(x_1,x_2) \in X \times X$, and the map 
$(x_1,x_2) \mapsto \lambda_{(x_1,x_2)}$ is continuous in the weak* topology. 
We define the measure $\sigma_a$ on $X\times X$ via
\begin{equation}\label{sigma}
\sigma_a= \int_Z \eta_z \times \eta_{2z-\pi(a)}\ dm(z) = \int_Z \eta_{\pi(a)+z} \times \eta_{\pi(a)+2z}\ dm(z).
\end{equation}
The previous can be found in \cite{kmrr1} for the case of $\N$ and in 
\cite{charamaras_mountakis2024} for general amenable groups $G$.

Let us denote by $\pi_1, \pi_2 \colon  X \times X \to X$ the
projections $(x_1,x_2) \mapsto x_1$, $(x_1,x_2) \mapsto x_2$ respectively
and by $\pi_i \sigma_a$ the pushforward of $\sigma_a$ under $\pi_i$, 
$i=1,2$.

\begin{proposition} \label{results from KMRR}
  Let $(x_1,x_2)$ be a point in $X \times X$ and $\lambda_{(x_1,x_2)}$, $\sigma_a$ be the measures on $X \times X$ defined respectively in \eqref{lambda} and \eqref{sigma}. Then  
    \begin{enumerate}
        \item \label{mu as averages of shifted measures} $\pi_1\sigma_a=\mu$ and $\frac{1}{\ell}\sum_{i=1}^\ell T_{{g}_i}\pi_2\sigma_a = \mu$. 
        \item \label{support and lambda(a,x)}
For $\sigma_a$-almost every $(x_1,x_2)\in X \times X$ we have that $(x_1,x_2) \in \textbf{supp}(\lambda_{(a,x_1)})$.

\item \label{generic point for lambda(a,x)} 
There exists 
a F\o lner sequence $\Psi$, such that for $\mu$-almost every $x_1\in X$ 
the point $(a,x_1)$ belongs to $\textbf{gen}(\lambda_{(a,x_1)},\Psi)$. 
    \end{enumerate}
\end{proposition} 

The fact that $\pi _1 \sigma_a =\mu$ is immediate from the definition.
The proofs of \eqref{support and lambda(a,x)} and \eqref{generic point for lambda(a,x)} can be found respectively in {\cite[Lemma 3.7 and Proposition 3.11]{kmrr2}} and \cite[Lemma 3.12]{kmrr2} for the case of $\N$, 
and in \cite[Theorem 4.9 and Lemma 4.14]{charamaras_mountakis2024} and 
\cite[Theorem 4.10 and Lemma 4.14]{charamaras_mountakis2024} for the case of more general groups $G$. Hence we only need to prove that 
$\frac{1}{\ell}\sum_{i=1}^\ell T_{{g}_i}\pi_2\sigma_a = \mu$.
To do this, we need the following lemma,
which asserts that we can find a subcollection $(g_{i_j})_j$ of 
$\{g_1, \ldots, g_{\ell}\}$ so that $Z$ can be written as a disjoint union 
of the cosets $\theta(g_{i_j}) + 2Z$, where $\theta \colon G \to Z$ is as in equation \eqref{eq kroneker}.

\begin{lemma}\label{Z as a union of cosets}
There exist an integer $k$ with $k\mid\ell$ and integers $i_1,\dots,i_k\in\{1,\dots,\ell\}$ such that 
\begin{equation}\label{Z coset representation}
    Z = \bigsqcup_{j=1}^k (\theta({g}_{i_j}) + 2Z).
\end{equation}
Moreover, for each $1\leq j\leq k$, we have $|\{1\leq i\leq\ell\colon 
\theta({g}_i)+2Z = \theta({g}_{i_j})+2Z\}| = \ell/k$.
\end{lemma}

\begin{proof}
From the assumptions on $G$ and $\theta$ we have that
$$Z = \overline{\theta(G)} = \bigcup_{i=1}^{\ell} \big(\theta({g}_i) + 2\overline{\theta(G)}\big)
= \bigcup_{i=1}^{\ell} (\theta({g}_i) + 2Z).$$
The sets $\theta({g}_i) + 2Z$ are cosets of $2Z$ in $Z$, hence any two such 
sets either coincide or they are disjoint. Consider $i_1,\dots,i_k\in\{1,\dots,\ell \}$ for some $1\leq k\leq\ell$
such that the cosets $\theta({g}_{i_j}) + 2Z$ are pairwise distinct, and then \eqref{Z coset representation} follows.
It remains to show that $k \mid \ell$ and that for each $1\leq j\leq k$, we have $|\{1\leq i\leq\ell\colon \theta({g}_i)+2Z = 
\theta({g}_{i_j})+2Z\}| = \ell/k$.
The homomorphism $\theta:G\to Z$ induces a homomorphism $\widetilde{\theta}:G/2G \to Z/2Z$, mapping ${g}_i+2G$ to $\theta({g}_i)+2Z$ for each $1\leq i\leq\ell$. It follows by \eqref{Z coset representation} that $\widetilde{\theta}$ is a surjective homomorphism of finite groups. Then, by the first isomorphism theorem, we have that 
$$\faktor{Z}{2Z} \simeq \faktor{\Big(\faktor{G}{2G}\Big)}{\ker(\widetilde{\theta})},$$
hence $k = \ell/|\ker(\widetilde{\theta})|$. This implies that $k\mid\ell$. Moreover, for fixed $1\leq j\leq k$, the above isomorphism implies that there exist exactly $|\ker(\widetilde{\theta})|=\ell/k$ integers $1\leq i\leq\ell$ such that $\theta({g}_i)+2Z = \theta({g}_{i_j})+2Z$. This concludes the proof.
\end{proof}

We are now ready to prove \eqref{mu as averages of shifted measures} of Proposition \ref{results from KMRR}.

\begin{proof}[Proof of \eqref{mu as averages of shifted measures} of Proposition \ref{results from KMRR}]
Let $k$ and $i_1,i_2,\dots,i_k$ as in Lemma \ref{Z as a union of cosets}. 
It is immediate from Lemma \ref{Z as a union of cosets} that $m(2Z)=1/k$. Now for each $u\in\{1, \ldots, \ell\}$ we define 
$$m_{2,u} = k\cdot m|_{2Z+\theta({g}_u)}.$$

Let $m_2$ denote the unique probability Haar measure on $2Z$. The pushforward 
$Dm$ of the Haar measure $m$ under the doubling map is translation invariant in
$2Z$, and it is a probability measure on $2Z$, so $m_2 = Dm$. Also, it is not 
difficult to see that the probability measure 
$k\cdot m|_{2Z}$ is also translation invariant in $2Z$, so after all, 
$m_2 = Dm=k\cdot m|_{2Z}$. Now for each $u$, let 
$D_u$ be the map sending $z$ to $2z+\theta({g}_u)$. Using the previous we then have 
that $m_{2,u} = D_um$. 

We first prove that for each $u\in\{1, \ldots,\ell\}$ we have
\begin{equation}\label{averaging the shifts of Haar}
    \frac{1}{k}\sum_{j=1}^k R_{{g}_{i_j}}m_{2,u} = m.
\end{equation}

Fix $u\in\{1,\ldots,\ell\}$. There exists a rearrangement $(i_j')_{1\leq j\leq k}$ of $(i_j)_{1\leq j\leq k}$ such that for each $j$ we have 
$2Z + \theta({g}_{i_j '}) = 2Z + \theta({g}_{i_j}) + \theta({g}_u)$.
Using Lemma \ref{Z as a union of cosets} and the definition of $m_{2,u}$, we have that for any measurable $C\subset Z$, 
\begin{align*}
    m(C) & = m\bigg(C\cap\bigg(\bigsqcup_{j=1}^k(\theta({g}_{i_j'}) + 2Z)\bigg)\bigg)
    = m\bigg(\bigsqcup_{j=1}^k(C\cap(\theta({g}_{i_j})+ \theta({g}_{u}) +2Z))\bigg) \\
    & 
    = \sum_{j=1}^k m(C\cap(\theta({g}_{i_j})+\theta({g}_u)+2Z))= 
    \sum_{j=1}^k m((C-\theta({g}_{i_j}))\cap (2Z+\theta({g}_u)))\\
    & = \frac{1}{k}\sum_{j=1}^k m_{2,u}(C-\theta({g}_{i_j}))
    = \frac{1}{k}\sum_{j=1}^k R_{{g}_{i_j}}m_{2,u}(C).
\end{align*}

Now we prove that for each $u\in\{1,\ldots,\ell\}$ we have
\begin{equation}\label{good averaging the shifts of Haar}
    \frac{1}{\ell}\sum_{i=1}^\ell R_{{g}_i} m_{2,u} = m.
\end{equation}

We fix $u\in\{1,\ldots,\ell\}$.
By the last assertion of Lemma \ref{Z as a union of cosets}, for each $1\leq j\leq k$, we can consider $1\leq i^{(j)}_1,\dots,i^{(j)}_{\ell/k}\leq \ell$ such that for each $1\leq v\leq\ell/k$, we have 
$\theta({g}_{i_j}) + 2Z = \theta({g}_{i^{(j)}_v}) + 2Z$, which implies that $R_{{g}_{i_j}}m_{2,u} = R_{{g}_{i^{(j)}_v}}m_{2,u}$. Moreover, \eqref{Z coset representation} implies that 
$\{\{i^{(j)}_v\colon 1\leq v\leq\ell/k\}\colon 1\leq j\leq k\}$ is a partition of $\{1,\dots\ell\}$. Combining all the above and using \eqref{averaging the shifts of Haar}, we have that 
$$m = \frac{1}{k}\sum_{j=1}^k {R_{{g}_{i_j}}}m_{2,u}
= \frac{1}{k}\sum_{j=1}^k \bigg(\frac{1}{\ell/k}\sum_{v=1}^{\ell/k}R_{{g}_{i^{(j)}_v}}m_2\bigg)
= \frac{1}{\ell}\sum_{i=1}^\ell R_{{g}_i}m_{2,u},$$
proving \eqref{good averaging the shifts of Haar}.

Now we can conclude the proof. 
In view of \eqref{Z coset representation}, there is $u\in\{1,\ldots,k\}$ and $z_0\in Z$ such that $\pi(a) = \theta(g_u)+2z_0$. First, we compute $\pi_2\sigma_a$ and using the translation invariance of $m$ we have that 
\begin{equation}\label{pi_2 sigma}
    \pi_2\sigma_a = \int_Z \eta_{2z+\pi(a)}\d m(z)
    = \int_Z \eta_{2(z+z_0)+\theta(g_u)}\d m(z)
    = \int_Z \eta_{2z+\theta(g_u)}\d m(z)
    = \int_Z \eta_z \d m_{2,u}(z),
\end{equation}
where the last equality follows from the fact that $m_{2,u} = D_um$.
Combining \eqref{good averaging the shifts of Haar} and \eqref{pi_2 sigma},
we obtain that
\begin{align*}
    \frac{1}{\ell}\sum_{i=1}^\ell T_{g_i}\pi_2\sigma_a
    & = \frac{1}{\ell}\sum_{i=1}^\ell \int_Z T_{g_i}\eta_z \d m_{2,u}(z)
    = \frac{1}{\ell}\sum_{i=1}^\ell \int_Z \eta_{R_{g_i}(z)} \d m_{2,u}(z)
    = \int_Z \eta_z \d \bigg(\frac{1}{\ell}\sum_{i=1}^\ell R_{g_i}m_{2,u}\bigg)(z) \\
    & = \int_Z \eta_z \d m(z) 
    = \mu. \qedhere
\end{align*}
\end{proof}

Using Proposition \ref{results from KMRR} we can guarantee the existence of many Erd\H{o}s 
progressions starting at the point $a$.

\begin{proposition}\label{EP full measure}
Let $(X,\mu,T)$ be an ergodic $G$-system and assume there is a 
continuous factor map $\pi\colon  X \to Z$ to its Kronecker factor. Let 
$a\in \textbf{gen}(\mu,\Phi)$, for some F\o lner sequence $\Phi$. 
Then for $\sigma_a$-almost every $(x_1,x_2) \in X \times X$, the point 
$(a,x_1,x_2)$ is an Erd\H{o}s progression. 
\end{proposition}

\begin{proof}
Let $\Psi$ be the F\o lner sequence and $L\subset X$ be the full 
$\mu$-measure set of $x\in X$ such that $(a,x_1)$ belongs to 
$\textbf{gen}(\lambda_{(a,x_1)},\Psi)$, arising from
\eqref{generic point for lambda(a,x)} of Proposition \ref{results from KMRR}. 
By \eqref{mu as averages of shifted measures} of Proposition \ref{results from KMRR}, 
we have that 
$\sigma_a(L\times X)=\mu(L)=1$ and so, for $\sigma_a$-almost every 
$(x_1,x_2) \in X\times X$ we have that $(a,x_1) \in 
\textbf{gen}(\lambda_{(a,x_1)},\Psi)$. 
In view of \eqref{support and lambda(a,x)} of Proposition \ref{results from KMRR},  
it follows that for $\sigma_a$-almost every $(x_1,x_2) \in X\times X$, 
\begin{enumerate}
    \item $(a,x_1) \in \textbf{gen}(\lambda_{(a,x_1)},\Psi)$ and 
    \item $(x_1,x_2) \in \textbf{supp}(\lambda_{(a,x_1)})$.  
\end{enumerate}
Thus, applying \cite[Lemma 2.5]{charamaras_mountakis2024}, 
we have that for $\sigma_a$-almost every $(x_1,x_2)\in X\times X$ the point 
$(a,x_1,x_2)$ is indeed an Erd\H{o}s progression.  
\end{proof}

\subsection{The proof of Theorem \ref{analogue of 2.1 KR}}\label{dynamical results}
We are now ready to prove Theorem \ref{analogue of 2.1 KR}. We first prove the 
following special case, and then explain how the general case follows from that.

\begin{theorem} \label{analogue of 3.4 KR}
Let $(X,\mu,T)$ be an ergodic $G$-system admitting a continuous factor map to its Kronecker factor, $a\in \gen(\mu,\Phi)$ for some \Folner{} sequence $\Phi$ in $G$,  
and $E_1,E_2 \subset X$ be open sets satisfying
\begin{equation}\label{1}
\ell\mu(E_2)+\mu(E_1) > \ell.
\end{equation}
Then, there exists an \Erdos{} progression $(a,x_1,x_2)$ such that $(x_1,x_2) \in E_1 \times E_2$.
\end{theorem}

\begin{proof}
From Proposition \ref{EP full measure}, for $\sigma_a$-almost every $(x_1,x_2) \in X \times X$,
the point $(a,x_1,x_2)$ is an Erd\H{o}s progression. 
Hence it suffices to verify 
that $\sigma_a(E_1 \times E_2)>0$. As 
$$E_1 \times E_2 = \left( E_1 \times X \right) \cap \left(  X\times E_2\right),$$
this reduces to showing that 
$$\sigma_a( E_1 \times X ) + \sigma_a \left(  X\times E_2\right) >1 .$$
Now, 
$$ \sigma_a(E_1 \times X)+\sigma_a(X \times E_2)=\pi_1\sigma_a(E_1)+\pi_2\sigma_a(E_2) $$
and by \eqref{mu as averages of shifted measures} of Proposition \ref{results from KMRR},
the right hand side equals 
$$\mu(E_1)+\ell \mu(E_2)-\pi_2\sigma_a(T^{-1}_{g_2}E_2)- \dots -
\pi_2\sigma_a(T^{-1}_{g_{\ell}}E_2).$$
Therefore, 
$$\sigma_a(E_1 \times X)+\sigma_a(X \times E_2) \geq \mu(E_1)+\ell \mu(E_2)-(\ell-1)>1,$$
where the last inequality follows by \eqref{1}. 
\end{proof}

Now let us see how Theorem \ref{analogue of 2.1 KR} follows from 
Theorem \ref{analogue of 3.4 KR}.

\begin{proof}[Proof of Theorem \ref{analogue of 2.1 KR}]
Let $(X,\mu,T)$ be an ergodic $G$-system, $\Phi$ a \Folner{} sequence in $G$, 
$a\in X$ such that $a\in \gen(\mu, \Phi)$, and $E_1, E_2\subset X$ open sets 
with $\ell \mu(E_2) + \mu(E_1)>\ell$. From 
\cite[Proposition 3.7]{charamaras_mountakis2024} we know that there exists 
an ergodic extension $(\widetilde{X},\widetilde{\mu},\widetilde{T})$ of 
$\xmt$, a \Folner{} sequence $\wt{\Phi}$ in $G$ and a point 
$\widetilde{a}\in\gen(\widetilde{\mu},\widetilde{\Phi})$ such that 
there exists a continuous factor map $\widetilde{\pi}:\widetilde{X}\to
X$ with $\widetilde{\pi}(\widetilde{a})=a$, 
$(\widetilde{X},\widetilde{\mu},\widetilde{T})$ has continuous factor map
to its Kronecker factor.

Let $\wt{E}_1=\wt{\pi}^{-1}(E_1)$ and $\wt{E}_2=\wt{\pi}^{-1}(E_2)$. 
From the definition of a factor map we know that $\wt{\pi} \wt{\mu}= \mu$ and so it follows that 
\begin{equation*}
\ell \wt{\mu}(\wt{E}_2) + \wt{\mu}(\wt{E}_1) >\ell. 
\end{equation*}
Since $(\widetilde{X},\widetilde{\mu},\widetilde{T})$ admits a continuous 
factor map to its Kronecker factor, we can use 
Theorem \ref{analogue of 3.4 KR} to find an Erd\H{o}s 
progression  $(\wt{a},\wt{x}_1,\wt{x}_2) \in \wt{X}^3$, such that $(\wt{x}_1,\wt{x}_2) \in \wt{E}_1 \times \wt{E}_2$. The continuity of 
$\wt{\pi}$ allows us to conclude that the triple
$(a, x_1,x_2):=(\wt{\pi}(\wt{a}),\wt{\pi}
(\wt{x}_1),\wt{\pi}(\wt{x}_2)) \in X^3$ is also an Erd\H{o}s 
progression in $(X,T)$ and clearly, by definition, $(x_1,x_2) \in E_1 \times E_2$.
\end{proof}

\section{The proof of Theorem \ref{optimality}: optimality of the lower bounds}\label{section_examples}

The goal of this section is to prove Theorem \ref{optimality}, so that the bounds we get from Theorem \ref{main_theorem_1} can be tested to be optimal. Indeed, for any possible value of $\ell$ and $r$, we provide a group $G$, a F\o lner sequence $\Phi=(\Phi_N)_{N\in \N}$ and a subset $A$ of that group the density of which achieves the bound established in Theorem 
\ref{optimality}, while not containing an infinite subset of the form $B + B$. 
 The main idea is to reverse-engineer the proof of Lemma \ref{correspondence} and construct a subset $A \subset G$ whose density achieves the bounds derived from that result.

For the group $\Z$, this was done in \cite[Section 4]{kousek_radic2024},
and is directly generalized to $\mathbb{Z}^d$ in 
Section \ref{section in Zd}. 
Note that in the previous cases we have $r=1$, where the doubling
map is expanding. However, the situation 
changes drastically when $r>1$. The key structural property is that, in 
this case, the doubling map is contracting. To capture both 
behaviors, in the general setting we work with groups of the form
$\mathbb{Z}^{d_1} \times (\mathbb{Z}(1/2)/\mathbb{Z})^{d_2}$.
On these groups, we consider \Folner{} sequences defined as unions of 
products of ``intervals'', reflecting the different behavior of the 
doubling 
map in each component. Finally, the set $A$ is constructed so that every 
element has at least one coordinate that is ``odd'', ensuring that it does 
not contain a set of the form $B+B$ for any infinite set $B$.

We highlight that as a consequence of Proposition \ref{prop equiv formulation}, if we show that 
one of the bounds in Theorem \ref{main_theorem_1} is sharp, then the other two are also 
sharp. Thus, for each value of $\ell$ and $r$ we shall provide an example that 
achieves optimal density for (\ref{main_theorem_1_1}) of Theorem \ref{main_theorem_1}.

\subsection{Case $\ell = 1$ and $r=1$} \label{sec ex F_p}

We recall that for a prime $p$, $\F_p = \Z / p\Z$.
We  construct this example in $\F_p^{\omega}$ for some odd prime $p$, for 
which there is a set $E_p \subset \F_p$ such that $E_p \cap 2 E_p = \emptyset$ and $|E_p|=(p-1)/2$. This is the case, for example, for $p=3$ with $E_3 = \{1\}$, and for $p=11 $ with $E_{11} = \{ 1,3,4,5,9\}$. However, there is no such set when $p=7$.\footnote{If $\pi_p: \F_p^X \to \F_p^X$ is the permutation induced by multiplication by $2$ in the group of units of $\F_p$, then such a set $E_p$ exists if and only if the disjoint decomposition of $\pi_p$ consists exclusively of odd cycles. For instance, $\pi_3 = (1~2)$, $\pi_{11} = (1~2~4~8~5~10~9~7~3~6)$, but $\pi_{7} = (1~2~4)(3~6~5)$.} Let $p$ be an odd prime with this property. 

We let $\Phi_N = \{ x \in \F_p^{\omega}: x_i = 0 \text{ for all } i > N \}$, $N\in \N$, $\Phi_0 = \{ 0 \}$, and by $e_N$ we denote the canonical vector, that is, $(e_N)_i=1$ if $i=N$ and  $(e_N)_i =0$ otherwise. Observe that $\Phi=(\Phi_N)_{N \in \N}$ is a F\o lner sequence in $\F_p ^{\omega}$ with $\alpha_\Phi = \alpha_{\mathbb{F}^{\omega}_p} = 1$.  
For each $N$, let $A_N= \bigsqcup_{i\in E_p}   \Phi_{N-1} + i \cdot e_{N}$, and take
\begin{equation} \label{eq def ex Fp}
    A = \bigsqcup_{N\geq 1}  A_N = \bigsqcup_{N\geq 1} \left( \bigsqcup_{i\in E_p}   \Phi_{N-1} + i \cdot e_{N} \right)
\end{equation}
Notice that 
$$\displaystyle \frac{\lvert A \cap \Phi_N\rvert}{\lvert \Phi_N\rvert} = \frac{p-1}{2} \cdot \frac{1 + p + \dots + p^{N-1} }{p^N} = \frac{p-1}{2} \left( \frac{1 - 1/p^N}{p-1} \right). $$ 
Therefore, taking limit as $N\to \infty$ we get $\diff_\Phi(A) = \frac{1}{2} = 1 - \frac{\alpha_{\F_p^{\omega}}}{\ell_{\F_p^{\omega}} + r_{\F_p^{\omega}}}$.

\begin{lemma} \label{lemma sharpness Fp}
    If $B+B \subset A$ for $A$ as in \eqref{eq def ex Fp} and $B \subset \F_p^{\omega}$, then $B$ is finite. 
\end{lemma}

\begin{proof}
Suppose there is an infinite $B \subset \F_p^{\omega}$ such that $B + B \subset A$. For $x \in \F_p^{\omega}$, we denote by $\iota(x)$ the greatest index $i$ such that $x_i \neq 0$. As $B$ is infinite, we can find $b,b'\in B$ such that $\iota(b') < \iota(b)$.
Since $b + b' \in A$, there exists $N \geq 1$ such that $b+b' \in A_N$. In particular, $\iota(b+b') = N$ and $ (b+b')_N \in E_p$. But since $\iota(b') < \iota(b) $, we have that $\iota(b) = N$ and $b_N \in E_p$. Finally, $\iota(2b) = N$, $2b_N \in 2 E_p$ which implies $\displaystyle 2b \in \bigcup_{i \in 2 E_p} (\Phi_{N-1} + i \cdot e_N) $. The latter set is disjoint from $A$, contradicting the fact that $2b \in A$.
\end{proof}

Lemma \ref{lemma sharpness Fp} concludes the proof of the bound's sharpness in $\F_p^{\omega}$. 

\begin{remark*}
Another natural group with $\ell=r=1$ is the rationals 
$(\Q,+)$. We were
able to construct a 
\Folner{} sequence $\Phi$ and a set $A\subset \Q$ such that $\alpha_{\Phi}=1$, $\underline{\diff}_{\Phi}(A)\geq \frac{1}{2}$ and $A\not\supset B+B$ for any infinite $B\subset \Q$, 
thus proving optimality of the bounds in 
Theorem \ref{main_theorem_1} for $(\Q,+)$. For brevity, we
decided to omit this example.
\end{remark*}

\subsection{Case $r=1$ and $\ell = 2^d,$ $d\geq 1$} \label{section in Zd}

The following examples are defined in $\Z^d$ for $d \geq 1$ and they are a direct generalization of the one given in \cite[Section 4]{kousek_radic2024} that was constructed in $\N$. For a vector $x=(x_1,\ldots,x_d) \in \Z^d$, we denote $\displaystyle \lVert x \rVert_{\infty} = \max_{i = 1, \ldots , d} |x_i|$.

\begin{lemma} \label{lemma ex from N to Zd}
    Let $A' \subset \N$ be such that, if $B'+B' \subset A' $ for some $B' \subset \N $, then $B'$ is finite. Consider $A =  \{ x \in \Z^d : \lVert x \rVert_{\infty} \in A' \}$. Then $B \subset \Z^d$ is such that $B + B \subset A$, only if $B$ is finite. 
\end{lemma}

\begin{proof}
    Reasoning by contradiction, suppose that there exists an infinite set $B \subset \Z^d$ such that $B + B\subset A$. By pigeonhole principle, there exists an infinite subset of $B$ (which we still denote by $B$), such that for every $x,y\in B$, $x_i \cdot y_i \geq 0$ for all $i \in \{1, \ldots, d\}$. Applying the pigeonhole principle again, we can assume there exists a fixed $i_0 \in \{1, \ldots , d\}$ such that $\lVert x \rVert_{\infty} = |x_{i_0}|$ for all $x \in B$. Now, choose two arbitrary elements $x,y \in B$. Since $x+y \in A$, there exists $a'\in A' $ such that
    $a' =|x_{i_0} + y_{i_0} | $, in particular, by symmetry of $A$ we can suppose that the $i_0$-th coordinate of every element in $B$ is nonnegative.  
    
    After these reductions, define $B_{i_0}=\{  x_{i_0} : x \in B\} \subset \N_0 $. By construction, $B_{i_0}  + B_{i_0} \subset A'$, and thus $B_{i_0}$ has to be finite. Since $B \subset \{ x : \lVert x \rVert_{\infty} \leq \max B_{i_0}\}$, this contradicts the infinitude of $B$.
\end{proof}

In this and the following sections we use the interval notation for discrete intervals, for example if $a,b \in \R$, we write $[a,b]$ for  $[a,b] \cap \Z$. Consider the set $A'=A''\cup (2\N+1)$, where
\begin{equation*}
    A'' = \bigcup_{n \in \N} [4^n, (2-1/n) \cdot 4^n).
\end{equation*}
It was shown in \cite{kousek_radic2024} that $A'$ does not contain 
an infinite sumset of the form $B'+B'$. Let $A_d = \{ x \in \Z^d \colon \lVert x \rVert_{\infty} \in A'\}$ and $A_d'=\{ x \in \Z^d \colon \lVert x \rVert_{\infty} \in A''\}$. A convenient way to represent $A_d'$ is via
\begin{equation*}
    A_d' = \bigsqcup_{n \in \N} (-(2-1/n) \cdot 4^n, (2-1/n) \cdot 4^n)^d \backslash (-4^n,4^n)^d.
\end{equation*}
Notice that $\Z^d$ has a natural F\o lner sequence given by $\Phi^d_N = [-N,N]^d$ for $N \in \N$ and that $\alpha_{\Phi^d} = \alpha_{\Z^d} = \frac{1}{2^d}$. Computing the density of $A_d'$ with respect to that F\o lner sequence we find
\begin{align*}
    \overline{\diff}(A_d')=  \lim_{N \to \infty} \frac{|[-(2-1/N)4^N, (2-1/N)4^N]^d \cap A_d'|}{|[-(2-1/N)4^N, (2-1/N)4^N]^d|}
    = 1 - \frac{1}{2^d +1} = 1 - \frac{\ell_{\Z^d} \alpha_{\Z^d}}{\ell_{\Z^d}+r_{\Z^d}}.
\end{align*}
Therefore, 
$$\overline{\diff}(A_d)=\overline{\diff}(A_d')+ \frac{2^d-1}{2^d}(1-\overline{\diff}(A_d'))=1-\frac{1}{2^d(2^d+1)}=1 - \frac{\alpha_{\Z^d}}{ \ell_{\Z^d}+r_{\Z^d}}.$$
Using Lemma \ref{lemma ex from N to Zd}, the set $A_d \subset \Z^d$ gives us the sharpness of the corresponding bound. 

\subsection{Case $\ell =1$ and $r=2^d,$ $d\geq 1$} \label{ex in Z(1/2) / Z}

In this subsection and the following ones we use the group of dyadic points in $\R/\Z$, that is,  $G=\Z(1/2) / \Z = \{ k /2^N \mod 1 : k, N \in \N \}$ and the disjoint family of subsets $(C_n)_{n \geq 0}$ given by 
\begin{equation} \label{eq family C_N}
    C_0 = \{0\}, C_1 = \{1/2\} \text{ and in general } C_{n} = \left\{ \frac{k}{2^n} \colon 0\leq k < 2^n, k \text{ odd} \right\} \text{ for } n \geq 1.
\end{equation} 
Notice that $\displaystyle G = \bigcup_{n \geq 0} C_n = \{ (2k+1) /2^n \in \Q/\Z : k, n \geq 0 \}$ and also that
\begin{equation} \label{eq C_N/2}
    C_0/2 = \{0,1/2\}= C_0 \cup C_1, \text{ and  } C_{n}/2 = C_{n+1} \text{ for } n \geq 1.
\end{equation} 
Equation \eqref{eq C_N/2} is a key feature for this and the following examples. With this family of sets, one can also describe a natural F\o lner sequence $F = (F_N)_{N \geq
0}$ in $G$ given by $F_N = \{ k/2^N \colon 0 \leq k < 2^N \} = \bigsqcup_{n = 0}^N 
C_n $. Using \eqref{eq C_N/2} we deduce that $F_N /2 = F_{N+1} = F_N \cup C_{N+1}$ and 
therefore $\alpha_{F} = 1$. 

\hfill

Consider the group $G^d$ and notice that $\ell_{G^d} = 1$, $r_{G^d} = 2^d$.
We construct something analogous to the example given in the previous subsection. Let $\Phi=(\Phi_N)_{N\in \N}$ be the \Folner{} sequence in 
$G^d$ defined by $\displaystyle \Phi_N = \underbrace{F_N 
\times \cdots \times F_N}_{d\text{ times}}$. As before, for each 
$N$, $\Phi_N /2 \supset \Phi_N$, which implies that 
$\alpha_{\Phi}=1$.

Before enouncing the construction we define some functions that are also useful in the following examples. Consider $\theta \colon G \to \N_0$ be the function such that 
\begin{equation} \label{eq degree function}
 \theta \left(\frac{2k+1}{2^n}\right)=n ~ \text{ for all } g= \frac{2k+1}{2^n} \in G.
\end{equation}
That is, for every $g \in G$, $\theta(g) = n$ if and only if $g \in C_n$.
Notice that, for $g, g' \in G$, if $\theta(g)\neq \theta(g')$,
then $\theta(g+g')=\max(\theta(g), \theta(g'))$, while if $\theta(g)=\theta(g')$, then 
$\theta(g+g')<\theta(g)$. Using the function $\theta \colon G \to \N_0$ we define
two functions $w\colon G^d \to N_0 $ and $\eta\colon G^d \to \{1, \ldots, d\}$ as 
follows: for $y = (y_1, \ldots, y_d) \in G^d$
\begin{align}
    w(y) =\max \{ \theta(y_j) \colon j =1, \ldots, d\} \label{eq norm index} \\
    \eta(y)=\min \{{i\in \{1, \dots, d\}} \colon \theta (y_i)= w(y)\} \label{eq index eta}
\end{align}
where $\eta(y)$ should be interpreted as the smallest index $i \in \{1, \ldots,d\}$ where 
the maximum defined in $ w(y)$ is achieved. For instance in $G^3$, 
$\eta\left( 0, \frac{1}{8}, \frac{3}{8} \right)= 2$.

To understand the following example, it might be useful to think that we try to replicate the example in $\Z^d$, where now the function $\theta \colon G \to \N_0$ plays an analogous role to the one of the absolute value and $w \colon G^d \to \N_0$ to the uniform norm. This analogy can in fact be made precise via the $2$-adic absolute value.

\begin{lemma} \label{lemma in z(1/2) about A}
Let $A \subset (\Z(\frac{1}{2})/\Z)^d$ be the set given by
    \begin{equation}
     A = \bigcup_{n \geq 0}  (F_{2n+1}  \times \cdots \times F_{2n+1}) \backslash (F_{2n}  \times \cdots \times F_{2n} ) = \bigcup_{n \geq 0} ( \Phi_{2n+1} \backslash \Phi_{2n}).
\end{equation}
If $B+B \subset A$ for some $B \subset G^d$, then $B$ is finite.
\end{lemma}

\begin{proof}
By contradiction, suppose that there exists an infinite set $B \subset G^d$ such that $B+B  \subset A$. Notice that, using \eqref{eq norm index}, we can rewrite $A$ as 
$$A =\{ y \in G^d \colon w(y) \text{ is odd} \}.$$ 
Since $B$ is infinite, without loss of generality, one can suppose that there exists $i \in \{1, \ldots, d\}$ such that for all $b \in B$, $\eta(b) = i$. Moreover, since  each $C_n$ is finite, we can find $b^{(1)}, b^{(2)} \in B$, such that 
$w(b^{(1)})<w(b^{(2)})$. 

Now, $2b^{(2)}\in A$,
and since $w(2b^{(2)})= w(b^{(2)})-1$ is odd, it follows 
that $w(b^{(2)})$ is even. 
But then, since $w(b^{(1)})<w(b^{(2)})$ and $\eta(b^{(1)}) = \eta (b^{(2)})$, we 
have that $w(b^{(1)}+b^{(2)}) = \theta(b^{(1)}_i+b^{(2)}_i) = \max\{\theta(b^{(1)}_i),\theta(b^{(2)}_i)\}=
\theta(b^{(2)}_i) = w(b^{(2)})$, which is even. But this means that 
$b^{(1)}+b^{(2)}$ is not in $A$, which is a contradiction. 
\end{proof}

We conclude the sharpness of the bound given by 
Theorem \ref{main_theorem_1} by computing the density of $A$. 
That is,
\begin{align*}
    \overline{\diff}_{\Phi}(A) &= \lim_{N \to \infty} \frac{\sum_{n=0}^N |(F_{2n+1}  \times \cdots \times F_{2n+1}) \backslash (F_{2n}  \times \cdots \times F_{2n} )|  }{|\Phi_{2N+1}|} \\
   &= \lim_{N \to \infty} \frac{\sum_{n=0}^N (2^{2n+1})^d - (2^{2n})^d  }{(2^{2N+1})^d } = \frac{2^d}{2^d +1}=
   1- \frac{\alpha_{G^d}}{\ell_{G^d} + r_{G^d}}.
\end{align*}

\subsection{Case $\ell = 2^{d_1}$ and $r = 2^{d_2}$ with $d_1 \geq d_2\geq 1$}\label{sec_ex_ell>r}

In this subsection, we construct the desired examples for general $\ell,r>1$ with $\ell \geq r$.
From Lemma \ref{powers_of_2} we know that for any countable abelian group
$G$, the quantities $r=|\ker(D)|, \ell=[G:2G]$ are integer powers of 
$2$. Let $d_1, d_2 \in \N$ with $d_1 \geq d_2$, and consider 
the group $G=\Z ^{d_1}\times (\Z(1/2) / \Z)^{d_2}$. Then 
$r=2^{d_2}$ and $\ell=2^{d_1}$ and $\ell \geq r$.
For an element $z\in G$, we write $z=(z^{(1)}, z^{(2)})$, 
where $z^{(1)}\in \Z^{d_1}$ and $z^{(2)}\in (\Z(1/2)/\Z)^{d_2}$.

Let $c(N),v(N)$ be two strictly increasing sequences of natural numbers
so that $c(N)$ is always even and $v(N+1)>v(N)+c(N)+1$ for all $N\in \N$ (for example, take $v(N)=3^N, c(N)=2N$). For $k \in \N$, we recall that $C_k \subset \Z(\frac{1}{2})/\Z$ is the set defined in \eqref{eq family C_N} and we also write $I_k = [-2^k, 2^k]\setminus \{0\} $. For each $N$, let 
\begin{equation} \label{eq Folner d_1 > d_2}
    \Phi_N= I_{c(N)}^{d_1}
\times C_{v(N)+1}^{d_2}\sqcup 
I_{c(N)-1}^{d_1}
\times C_{v(N)+2}^{d_2}\sqcup
\cdots \sqcup 
I_1^{d_1} \times C_{v(N)+c(N)}^{d_2}
= \bigsqcup_{m=0}^{c(N)-1} 
I_{c(N)-m}^{d_1}
\times C_{v(N)+m+1}^{d_2}.
\end{equation}

We highlight that, in the definition of $I_k$ we remove $0$ from the discrete interval $[-2^{k}, 2^{k}]$ only to simplify the computations in what follows, but the same result is true if we do not remove it.

As $v(N)+c(N)<v(N+1)$, the $\Phi_N$'s are pairwise disjoint. We 
prove that $(\Phi_N)_{N\in \N}$ is a \Folner{} sequence in $G$, and 
that $\alpha_{\Phi}=\frac{r}{\ell}=\min\{1,\frac{r}{\ell}\}$. 

For each $N$, we see that
\begin{equation}\label{useful_calc_1}
    |\Phi_N| = \begin{cases}
        c(N) 2^{d_1 c(N) + d_1 v(N) + d_1}, \text{ if } d_1=d_2 \\
        2^{d_1 c(N) + d_1 + d_2 v(N)}\cdot \frac{1-( 2^{d_2 -d_1} )
        ^{c(N)}}{1-2^{d_2-d_1}}, \text{ if } d_1\neq d_2 .
    \end{cases}
\end{equation}

Let $(x,y)\in G$, where $x=(x_1, \ldots, x_{d_1})$ and 
$y=(y_1, \ldots, y_{d_2})$. Recall the definition of $w(y)$ given in \eqref{eq norm index}.
Then for all integers $k_1, \ldots, k_{d_2} > w(y)$, $y+ C_{k_1} \times 
\cdots \times C_{k_{d_2}} = C_{k_1} \times \cdots \times C_{k_{d_2}}$. In particular if $v(N) > w(y)$, then $C_{v(N) + m+1}^{d_1} + y =C_{v(N) + m+1}^{d_1}$ for every $m \in \N$.

By construction, $v(N)\geq N$, so that for $N>w(y)$,
$v(N)>w(y)$, and therefore 
$$\big((x,y)+\Phi_N\big) \triangle \Phi_N=\bigsqcup_{m=0}^{c(N)-1}
\Big( \big( x+ I_{c(N)-m} ^{d_1} \big)
\triangle
I_{c(N)-m} ^{d_1} \Big) 
\times C_{v(N)+m+1}^{d_2}. $$
Computing the cardinality we get
\begin{equation*}
\big|\big( x+ I_{c(N)-m} ^{d_1} \big)
\triangle I_{c(N)-m} ^{d_1} \big| 
\leq 2^{d_1} \big| (x+[1, 2^{c(N)-m}] ^{d_1}) \triangle 
[1, 2^{c(N)-m}] ^{d_1} \big|\leq 2^{d_1} \big(2^{c(N)-m}\big)^{d_1 -1}
\sum_{i=1}^{d_1} 2|x_i|,
\end{equation*}
and therefore
$$ \big| \big((x,y)+\Phi_N\big) \triangle \Phi_N \big| \leq \Big(\sum_{i=1}^{d_1} 2|x_i| \Big) 2^{d_1+d_1 c(N) +d_2 v(N) -c(N)} \sum_{m=0}^{c(N)-1} 
     \big(2^{d_2 +1 -d_1}\big)^{m}. $$
Finally,
\begin{equation}\label{foln_calc_1}
    \big| \big((x,y)+\Phi_N\big) \triangle \Phi_N \big| \leq \left\{
\begin{array}{cc}
     c(N) \big(\sum_{i=1}^{d_1} 2|x_i| \big) 2^{d_1+d_1 c(N) +d_2 v(N) -c(N)},&   \text{ if } d_1=d_2 +1 \\
      \big(\sum_{i=1}^{d_1} 2|x_i| \big) 2^{d_1+d_1 c(N) +d_2 v(N) -c(N)}\cdot \frac{1-(2^{d_2+1 -d_1})^{c(N)}}{1-2^{d_2+1 -d_1}}, 
      &\text{ if } d_1\neq d_2 +1.
\end{array}         \right.
\end{equation}
Using \eqref{useful_calc_1} and \eqref{foln_calc_1}, we can prove the following
lemma:
\begin{lemma} \label{lemma Phi is a Folner in d_1 d_2}
    For all $d_1\geq d_2$, $(\Phi_N)_{N \in \N}$ defined in \eqref{eq Folner d_1 > d_2} is a F\o lner sequence in $\Z^{d_1} \times (\Z(1/2)/\Z)^{d_2}$.
\end{lemma}

\begin{proof}
    
If $d_1 =d_2$, then 
\begin{equation*}
    \frac{| ((x,y)+\Phi_N) \triangle \Phi_N |}{|\Phi_N|}\leq 
    \frac{\big(\sum_{i=1}^{d_1} 2|x_i| \big) 2^{d_1+d_1 c(N) +d_1 v(N) -c(N)}(2^{c(N)}-1)}{c(N) 2^{d_1 c(N) + d_1 v(N)+ d_1}}=
    \frac{\sum_{i=1}^{d_1} 2|x_i|}{c(N)} (1-2^{-c(N)})\xrightarrow{N \to \infty} 0.
\end{equation*}
Similarly, if $d_1 =d_2 +1$, then 
\begin{equation}\label{useful_eq_1}
    \frac{| ((x,y)+\Phi_N) \triangle \Phi_N |}{|\Phi_N|}\leq 
    \frac{c(N)\big(\sum_{i=1}^{d_1} 2|x_i| \big)
    2^{d_2 +1 + d_2 c(N) + d_2 v(N)}}{(1-2^{-c(N)})2^{d_2 c(N)+c(N) 
    +d_2 v(N) +d_2 +2}}=
    \frac{\sum_{i=1}^{d_1} 2|x_i|}{1-2^{-c(N)}}\cdot 
    \frac{c(N)}{2^{c(N)+1}}. 
\end{equation}
Since $\lim_{h\to \infty} \frac{h}{2^{h+1}}=0$ and $c(N)\to \infty$, we 
have that $\frac{c(N)}{2^{c(N)+1}} \to 0$ as $N\to \infty$. Thus, \eqref{useful_eq_1} immediately implies that $\frac{| ((x,y)+\Phi_N) \triangle \Phi_N |}{|\Phi_N|} \to 0$ as $N\to \infty$. 
The proof when $d_1\geq d_2+2$ is analogous and is omitted.
\end{proof}

\begin{lemma}
    For all $d_1\geq d_2$, the F\o lner sequence $(\Phi_N)_{N \in \N}$ defined in \eqref{eq Folner d_1 > d_2} has ratio $\alpha_{\Phi}=\alpha_G=
    2^{d_2-d_1}$.
\end{lemma}

\begin{proof}
    For each $N$,
\begin{equation*}
    \Phi_N /2 = 
    \bigsqcup_{m=0}^{c(N)-1} 
I_{c(N)-m-1}^{d_1}
\times C_{v(N)+m+2}^{d_2} =\bigsqcup_{m=1}^{c(N)} 
I_{c(N)-m}^{d_1}
\times C_{v(N)+m+1}^{d_2},
\end{equation*}
so $\Phi_N \setminus (\Phi_N/2)=I_{c(N)}^{d_1} \times C_{v(N)+1}^{d_2}$.

If $d_1 =d_2$, then 
\begin{equation*}
    \frac{|\Phi_N \setminus (\Phi_N/2)|}{|\Phi_N|}=
    \frac{2^{d_1 (c(N)+1)+d_1 v(N)}}{c(N) 2^{d_1 c(N) + d_1 v(N) + d_1}}
    = \frac{1}{c(N)}\xrightarrow{N \to \infty} 0
\end{equation*}
Therefore, 
$\frac{|\Phi_N \cap (\Phi_N/2)|}{|\Phi_N|}=
    1-\frac{|\Phi_N \setminus (\Phi_N/2)|}{|\Phi_N|}\to 1=
    \frac{r}{\ell}$ as $N\to \infty$.

Similarly, if $d_1 > d_2$, then
\begin{equation*}
    \frac{|\Phi_N \setminus (\Phi_N/2)|}{|\Phi_N|}=
    \frac{2^{d_1 (c(N)+1)+d_2 v(N)}}{2^{d_1 c(N) + d_1 + d_2 v(N)}\cdot \frac{1-( 2^{d_2 -d_1} )
        ^{c(N)}}{1-2^{d_2-d_1}}}=
     \frac{1-2^{d_2-d_1}}{1-( 2^{d_2 -d_1} )
        ^{c(N)}} \xrightarrow{N \to \infty} 1 - 2^{d_2 -d_1},
\end{equation*}
where in the final limit we use that $\displaystyle\lim_{N \to \infty} 2^{(d_2 -d_1)c(N)} = 0$. Thus, 
$\frac{|\Phi_N \cap (\Phi_N/2)|}{|\Phi_N|}\to 2^{d_2 -d_1}=
\frac{r}{\ell}$ as $N\to \infty$.
\end{proof}

To conclude this section, we build a set $A\subset G$ with $\diff_{\Phi}(A)=
1-\frac{r}{\ell (\ell+r)} = 1-\frac{\alpha_{\Phi}}{\ell+r}$ so that $A$ contains no $B+B$ for some infinite $B$. We highlight that for this example the density of $A$ along $\Phi$ exists. 

Recall that for each $N\in \N$, by construction, $c(N)$ is even. We let 
\begin{equation*}
    A_{2,N}= \bigsqcup_{m=0}^{\frac{c(N)}{2}-1} \big(I_{c(N)-2m} \cap 2\Z \big)^{d_1}\times C_{v(N)+2m+1}^{d_2}.
\end{equation*}

Observe that the $A_{2,N}$'s are pairwise disjoint. Let 
$A_2=\bigsqcup_{N\in \N}
A_{2,N}$. We also define a useful set that is used in this and in the next example,
\begin{equation} \label{eq at least one odd}
    O_{d_1}= (\Z^{d_1} \setminus (2\Z)^{d_1}) = 
\{(x_1, \ldots, x_{d_1}) \in \Z ^{d_1} \colon \text{ at least one } 
x_i \text{ is odd}\}
\end{equation}
and then $A_1 = O_{d_1} \times (\Z(1/2) / \Z)^{d_2}$.
Finally, take $A=A_1 \sqcup A_2$. Notice that
\begin{equation}\label{useful_eq_3}
    A\cap \Phi_N = 
    \Bigg[\bigsqcup_{m=0}^{c(N)-1} I_{c(N)-m}^{d_1} \cap 
O_{d_1}
\times C_{v(N)+m+1}^{d_2}\Bigg]
\sqcup A_{2,N}. 
\end{equation}
We provide a figure below to illustrate $\Phi_N$ (in black) and 
$A\cap\Phi_N$ (in red) in the case $d_1=d_2=1$, that is $G=\Z \times 
\Z(1/2)/\Z$. The dotted lines indicate that we only take 
elements with odd $\Z$-coordinate.

\begin{center}
\begin{tikzpicture}[baseline]
\begin{axis}[
    axis lines=middle,
    title={$\Phi_N$},
    xlabel={$\Z$},
    ylabel={$\Z(1/2)/\Z$},
    xtick={-9.6, -6.0, -3.8, -2, -1.2, -0.5, 0, 0.5, 1.2, 2, 3.8, 6.0, 9.6}, 
    xticklabels={$-2^{c(N)}$, $-2^{c(N)-1}$, $-2^{c(N)-2}$, $\dots$, $-4$, $-2$, $0$, $2$, $4$, $\dots$, $2^{c(N)-2}$, $2^{c(N)-1}$, $2^{c(N)}$}, 
    ytick={0,3,4,5,6.25,7.5,8.5},
    yticklabels={$0$, $C_{v(N)+1}$, $C_{v(N)+2}$, $C_{v(N)+3}$, $\dots$, $C_{v(N)+c(N)-1}$, $C_{v(N)+c(N)}$},
    ymin=0, ymax=10, 
    xmin=-10, xmax=10, 
    width=17cm, 
    height=8cm, 
    ticklabel style={font=\small}, 
    yticklabel style={yshift=0.2cm} 
]

    \addplot[domain=-9.6:-0.25, samples=2, thick] 
    {2.95};
    \addplot[domain=-9.6:-0.25, samples=2, thick, color=red] 
    {3.05};
    \addplot[domain=0.25:9.6, samples=2, thick] 
    {2.95};
    \addplot[domain=0.25:9.6, samples=2, thick, color=red] 
    {3.05};
    \addplot[domain=-6.0:-0.25, samples=2, thick] 
    {3.95};
    \addplot[domain=-6.0:-0.25, samples=2, line width=1pt, 
    dash pattern=on 2pt off 2pt, color=red] 
    {4.05};
    \addplot[domain=0.25:6.0, samples=2, thick] 
    {3.95};
    \addplot[domain=0.25:6.0, samples=2, line width=1pt,      dash pattern=on 2pt off 2pt, color=red] 
    {4.05};
    \addplot[domain=-3.8:-0.25, samples=2, thick] 
    {4.95};
    \addplot[domain=-3.8:-0.25, samples=2, thick, color=red] 
    {5.05};
    \addplot[domain=0.25:3.8, samples=2, thick] 
    {4.95};
    \addplot[domain=0.25:3.8, samples=2, thick, color=red] 
    {5.05};
    \addplot[domain=-1.2:-0.25, samples=2, thick] 
    {7.45};
    \addplot[domain=-1.2:-0.25, samples=2, thick, color=red] 
    {7.55};
    \addplot[domain=0.25:1.2, samples=2, thick] 
    {7.45};
    \addplot[domain=0.25:1.2, samples=2, thick, color=red] 
    {7.55};
    \addplot[domain=-0.5:-0.25, samples=2, thick] 
    {8.45};
    \addplot[domain=-0.25:-0.5, samples=2, line width=1pt,      dash pattern=on 2pt off 2pt, color=red] 
    {8.55};
    \addplot[domain=0.25:0.5, samples=2, thick] 
    {8.45};
    \addplot[domain=0.25:0.5, samples=2, line width=1pt,      dash pattern=on 2pt off 2pt, color=red] 
    {8.55};
    
\end{axis}
\end{tikzpicture}
\end{center}
\vspace{3mm}

With this formula we can compute the density of $A$. First notice that for each $m\in \{0, 1, \ldots, c(N)-1\}$, 
\begin{equation} \label{eq 2^d_1 - 1 / 2 ^d_1}
    \frac{\big|I_{c(N)-m}^{d_1} \cap O_{d_1}\big|}{\big|I_{c(N)-m}^{d_1}\big|}=
1- \frac{\big|I_{c(N)-m}^{d_1} 
    \cap 2 \Z^{d_1} \big|}
{\big| I_{c(N)-m}^{d_1}\big|}=
\frac{2^{d_1}-1}{2^{d_1}}.
\end{equation}
Therefore,
\begin{equation}\label{useful_eq_4}
\frac{\big|\bigsqcup_{m=0}^{c(N)-1}  \big(
I_{c(N)-m}^{d_1} \cap O_{d_1} \big)
\times C_{v(N)+m+1}^{d_2}\big|}{|\Phi_N|} \nonumber 
=\frac{2^{d_1}-1}{2^{d_1}}. 
\end{equation}

On the other hand,  
$|A_{2,N}|= \left\{ \begin{array}{cc}
    \frac{c(N)}{2} 2^{d_1c(N) + d_1 v(N)}& \text{ if } d_1 =d_2 \\
    2^{d_1c(N) + d_2 v(N)} \cdot \frac{1-2^{(d_2 -d_1)c(N)}}{1-2^{2(d_2 -d_1)}}& \text{ if } d_1 \neq d_2 
\end{array} \right.$

Therefore, if $d_1 =d_2$, using \eqref{useful_calc_1}, 
\eqref{useful_eq_3}, 
\eqref{eq 2^d_1 - 1 / 2 ^d_1} and the previous calculation for $|A_{2,N}|$ one 
sees that 
\begin{equation*}
\frac{|A\cap \Phi_N|}{|\Phi_N|} = \frac{2^{d_1}-1}{2^{d_1}} + 
\frac{1}{2^{d_1 +1}}=
\frac{2^{d_1+1}-1}{2^{d_1+1}}=1-\frac{\alpha_{\Phi}}{\ell+r}.
\end{equation*}
A similar calculation when $d_1 \neq d_2$ shows that 
$\diff_{\Phi}(A)=1-\frac{\alpha_{\Phi}}{\ell+r}.$

Now, it remains to prove that there is no infinite $B$ so that 
$B+B\subset A$. Assume that there is some infinite $B$ so that 
$B+B\subset A$. 
Since $B$ is infinite, we may assume that if $b = (x,y) $ and $b'= (x',y')$ are any two elements in $B$, then $x_i = x'_i \mod 2$ for all $i = 1, \ldots, d_1$.

Fix $b =(x,y) \in B$. Then by assumption $2b \in A$ and also $2b \in (2\Z)^{d_1} \times (\Z(1/2)/\Z)^{d_2}$.
Therefore $2b \in A_2$, in particular there is unique
$N_1 \in \N$ so that $2b \in A_{2,N_1}$. Since $B$ is infinite and 
$\bigcup_{j\leq N_1} A_{2,j}$ is finite, there exists 
$b'=(x',y') \in B$ so that $2b' \in A_{2,N_2}$ for some $N_2 > N_1$. 

Then, from the definition of the $A_{2,N}$'s we see that there are
$m_1\in \{0,1,\ldots, \frac{c(N_1)}{2}-1\} $ and 
$m_2\in \{0,1,\ldots, \frac{c(N_2)}{2}-1\} $ so that
$2b\in \big(I_{c(N_1)-2m_1}\cap 2\Z \big)^{d_1}
\times C_{v(N_1)+2m_1+1} ^{d_2}$
and 
$2b'\in \big(I_{c(N_2)-2m_2}\cap 2\Z \big)^{d_1}
\times C_{v(N_2)+2m_2+1} ^{d_2}$.
From the previous we infer that 
\begin{align*}
    x\in I_{c(N_1)-2m_1-1}^{d_1}, ~ y\in C_{v(N_1)+2m_1+2} ^{d_2}, ~
    x'\in I_{c(N_2)-2m_2-1}^{d_1} ~\text{ and } ~ y'\in C_{v(N_2)+2m_2+2} ^{d_2}.  
\end{align*}

By the parity assumption in $B$, $x + x' \in (2\Z)^{d_1}$,
so from the definition of $A$ and since 
$b + b' \in A$, we obtain that $b +b' \in A_2$. 
Now since $v(N_2)+2m_2 +2 > v(N_1)+2m_1 +2$ and using the properties of the function $\theta \colon \Z(1/2)/\Z \to \N_0 $ defined in \eqref{eq degree function}, we have that for all $j = 1, \ldots, d_2$, $\theta(y_j + y'_j) = \max\{ \theta(y_j), \theta(y'_j)\} =  \max\{v(N_1)+2m_1 +2, v(N_2)+2m_2 +2 \} $ and therefore
$y_j + y'_j \in C_{v(N_2)+2m_2+2}^{d_2}$.  

Thus $b + b' = (x+x', y+y') \in (2\Z)^{d_1} \times C_{v(N_2)+2m_2+2}^{d_2}$,
which is disjoint from $A_2$. This is a contradiction and therefore, 
there is no infinite $B$ so that $B+B \subset A$, concluding the construction. 

\subsection{Case $\ell = 2^{d_1}$ and $r = 2^{d_2}$ with $1\leq d_1 < d_2$}
\label{sec_ex_ell<r}

In this subsection, for $d_1, d_2 \in \N$ with $d_1 < d_2$, we consider again the group $G=\Z ^{d_1}\times (\Z(1/2) / \Z)^{d_2}$, so $\ell=2^{d_1}$ and $r=2^{d_2}$, but this time we have $\ell < r$.
As before, for every $z\in G$, we write $z=(z^{(1)}, z^{(2)})$, where $z^{(1)}\in \Z^{d_1}$ and $z^{(2)}\in (\Z(1/2)/\Z)^{d_2}$. For constructing the corresponding F\o lner sequence and the set $A$, again we use the sets $C_k \subset \Z(1/2)/\Z$ defined in \eqref{eq family C_N}, $I_k = [-2^k,2^k]\setminus\{0\} \subset \Z $ and $O_{d_1} \subset \Z^{d_1}$ as in \eqref{eq at least one odd}.

For this example, again we let $c(N),v(N)$ be two strictly increasing sequences of 
natural numbers with similar properties. In particular, $c(N)$ is even 
and $v(N+1)-c(N+1)>v(N)+1$ for all $N\in \N$. In this case we still can take $v(N)=3^N$ and $ c(N)=2N$. In this subsection, many computations are omitted, as they are very similar to the ones carried out in Section \ref{sec_ex_ell>r}.

The first major change is the definition of the \Folner{} sequence. For each $N$, set 
\begin{equation} \label{eq Folner d_1 < d_2}
\Phi_N= I^{d_1}_{c(N)}
\times C_{v(N)}^{d_2}\sqcup 
I_{c(N)+1}^{d_1}
\times C_{v(N)-1}^{d_2}\sqcup 
\cdots \sqcup 
I_{2c(N)-1}^{d_1} \times C_{v(N)-c(N)+1}^{d_2}
= \bigsqcup_{m=0}^{c(N)-1} 
I_{c(N)+m}^{d_1}
\times C_{v(N)-m}^{d_2}.
\end{equation}

As $v(N+1)-c(N+1)>v(N)+1$, the $\Phi_N$'s are pairwise disjoint. The proof that $(\Phi_N)_{N\in \N}$ is a \Folner{} sequence in $G$ is analogous to the one carried out in 
Lemma \ref{lemma Phi is a Folner in d_1 d_2}, so it is omitted. Similarly 
to Section \ref{sec_ex_ell>r} we have

\begin{lemma}
    For all $d_1 < d_2$, the F\o lner sequence $(\Phi_N)_{N \in \N}$ defined in \eqref{eq Folner d_1 < d_2} has ratio $\alpha_{\Phi} = 1 $.
\end{lemma}

\begin{proof}
For each $N$, $\displaystyle \Phi_N /2 = \bigsqcup_{m=0}^{c(N)-1} 
I_{c(N)+m-1}^{d_1} \times C_{v(N)-m+1}^{d_2} $, therefore $ \displaystyle \Phi_N \setminus (\Phi_N/2)=
I_{2c(N)-1}^{d_1} \times C_{v(N)-c(N)+1}^{d_2}$.
Performing similar computations to the ones in Section \ref{sec_ex_ell>r}, one sees that 
\begin{equation}\label{1_useful_calc_1}
  |\Phi_N|=
  2^{d_1 c(N) + d_1 + d_2 v(N)-d_2}\cdot \frac{1-( 2^{d_1 -d_2} )
        ^{c(N)}}{1-2^{d_1-d_2}},
\end{equation}
and thus
\begin{equation*}
    \frac{|\Phi_N \setminus (\Phi_N/2)|}{|\Phi_N|}=
    \big(1-2^{d_1 -d_2}\big) 2^{d_1 -d_2}
    \frac{2^{(d_1 -d_2)(c(N)+1)}}{1-2^{(d_1 -d_2)c(N)}}\xrightarrow{N \to \infty} 0,
\end{equation*}
 because $2^{(d_1 -d_2)c(N)} \to 0$ as $N \to \infty$. Therefore $\alpha_{\Phi} = 1$. 
\end{proof}

As in Section \ref{sec_ex_ell>r}, we end this section by building
a set $A\subset G$ such that $\diff_{\Phi}(A)= 1-\frac{1}{\ell+r} = 
1-\frac{\alpha_{\Phi}}{\ell+r}$ and if $B+B \subset A$, then $B$ is finite. The construction 
of $A$ is also similar. First, for all $N \in \N$ set 
\begin{equation*}
    A_{2,N}= \bigsqcup_{m=0}^{\frac{c(N)}{2}-1} \big( I_{c(N)+2m}\cap 2\Z \big)^{d_1}\times 
    C_{v(N)-2m}^{d_2}. 
\end{equation*}

Observe that the $A_{2,N}$'s are pairwise disjoint, so conveniently we define 
$A_2=\bigsqcup_{N\in \N}
A_{2,N}$. Again, $A_1 = O_{d_1} \times (\Z(1/2) / \Z)^{d_2}$ and $A=A_1 \sqcup A_2$. The proof that $A$ does not contain an infinite sumset $B+B$ is completely analogous to the one in the previous section, so we only compute the density where the computations are slightly different. Notice that,
\begin{equation}\label{1_useful_eq_3}
    A\cap \Phi_N = 
    \Bigg[\bigsqcup_{m=0}^{c(N)-1} 
(I_{c(N)+m}^{d_1} \cap O_{d_1})
\times C_{v(N)-m}^{d_2}\Bigg]
\sqcup A_{2,N}. 
\end{equation}

For every $m\in \{0, 1, \ldots, c(N)-1\}$, as in \eqref{eq 2^d_1 - 1 / 2 ^d_1}, $\displaystyle \frac{|I_{c(N)+m} ^{d_1} \cap O_{d_1}|}{|I_{c(N)+m} ^{d_1}|} = \frac{2^{d_1}-1}{2^{d_1}}$
and therefore 
\begin{equation}\label{1_useful_eq_4}
\frac{\big|\bigsqcup_{m=0}^{c(N)-1} 
\big(I_{c(N)+m} ^{d_1} \cap O_{d_1} \big) 
\times C_{v(N)-m}^{d_2}\big|}{|\Phi_N|} =\frac{2^{d_1}-1}{2^{d_1}}. 
\end{equation}

In addition, 
$|A_{2,N}|=2^{d_1c(N) + d_2 v(N) - d_2} \cdot 
\frac{1-2^{(d_1 -d_2)c(N)}}{1-2^{2(d_1 -d_2)}}$. Therefore, using \eqref{1_useful_calc_1}, \eqref{1_useful_eq_3}, 
\eqref{1_useful_eq_4} and the expression for $|A_{2,N}|$, one 
sees that 
\begin{equation*}
\frac{|A\cap \Phi_N|}{|\Phi_N|} = \frac{2^{d_1}-1}{2^{d_1}} + 
\frac{2^{d_2 -d_1}}{2^{d_1}+2^{d_2}}= 1-\frac{1}{\ell+r}
\end{equation*}
Hence, the density $\diff_{\Phi}(A)$ exists and is $1-\frac{1}{\ell+r}$, concluding the construction of the example in the case  $\ell < r$. 

\begin{remark*}
We note here that in both Sections \ref{sec_ex_ell>r} and \ref{sec_ex_ell<r} 
the \Folner{} sequences have the same triangular shape: they are of the form 
$\Phi_N=\bigsqcup_{m} I_{c_1(m)}^{d_1} \times C_{c_2(m)}^{d_2}$, where as
$c_2(m)$ grows larger -- so we take more elements from $(\Z(1/2)/\Z)^{d_2}$ -- 
$c_1(m)$ becomes smaller -- so we have fewer elements from $\Z^{d_1}$.
However, the \Folner{} sequence $\Phi$ defined in Section \ref{sec_ex_ell>r} is not 
a \Folner{} sequence in the case $\ell>r$ (because then $2^{d_2-d_1}<1$). 
This is the reason why we have to consider different $\Phi$'s in the two cases.
\end{remark*}

\section{Necessity of the assumptions in the main theorem}\label{section_necessity}

Herein, we show that the assumptions on the group and the \Folner{} 
sequence in Theorem \ref{main_theorem_1} are necessary.

\subsection{The kernel of the doubling map has to be finite}
\label{finite_kernel_is_necessary}

We first prove that the assumption $r=|\ker(D)|<\infty$ is 
necessary.

\begin{proposition}
    There is a group $G$ with $r=|\ker(D)|=\infty$, a \Folner{} sequence 
    $\Phi$ in $G$ with $\alpha_{\Phi}=1$, and a set $A\subset G$ with 
    $\overline{\textup{d}}_{\Phi}(A)=1$ so that for any infinite $B\subset G$ 
    we have $B+B \not \subset A$.
\end{proposition}

\begin{proof}
Let $G=\Z (\frac{1}{2})/\Z$ and consider the group $G^{\omega}$.
Then $\ell_{G^{\omega}}=1, r_{G^{\omega}}=\infty$. 
Recall the definition of $C_n$ given in \eqref{eq family C_N}. We have 
$|C_0|=1$, $|C_n|=2^{n-1}$ for $n\in \N$, the
sets $(C_n)_{n\in \N_0}$ are pairwise disjoint and 
$G=\bigcup_{n\in \N_0} C_n$. 
Also, the sequence $F=(F_N)_{N\in \N}$ defined by 
$F_N=\bigcup_{0\leq n \leq N} C_n$ is a \Folner{} sequence
in $G$ and $F_N/2 = F_{N+1} = F_N \cup C_{N+1} \supset F_N$,
so $\frac{|F_N/2\cap F_N|}{|F_N|}=1$ and hence
$\alpha_F=1$. Also, for all $N\in \N$, $|F_N|=2^N$.

Now, consider the sequence $(\Phi_N)_{N\in \N}$ of subsets of 
$G^{\omega}$ defined by 
\begin{equation*}
    \Phi_N = \underbrace{F_N\times \dots \times F_N}_{N 
    \text{ times }} \times \{0\}^{\omega} =\{ \boldsymbol{g} = (g_i)_{i\in \N} \colon g_i \in F_N 
    \text{ for } 1\leq i \leq N, g_i=0 \text{ for } i>N \}.
\end{equation*}
Then $\Phi=(\Phi_N)_{N\in \N}$ is a \Folner{} sequence in 
$G^{\omega}$. This F\o lner sequence shares some properties with the one in 
$G^d$ studied in Section \ref{ex in Z(1/2) / Z}. 
In particular,  
\begin{equation*}
    \Phi_N /2 = \underbrace{F_N /2\times \dots \times F_N/2}_{N 
    \text{ times }} \times \left\{0, 1/2 \right\}^{\omega} = \underbrace{F_{N+1}\times \dots \times F_{N+1}}_{N 
    \text{ times }} \times \left\{0, 1/2 \right\}^{\omega},
\end{equation*}
and therefore, $ \Phi_N \subset \Phi_N /2$, which implies that $\alpha_{\Phi}=1$. 

Consider now the set 
\begin{align*}
A = \bigsqcup_{m \in \N}  \left[
\big(\underbrace{F_{2m + 1} \times \dots \times F_{2m + 1}}_{2m +1
\text{ times }} \times \{0\}^{\omega} \big)
\setminus 
\big(\underbrace{F_{2m} \times \dots \times F_{2m}}_{2m +1
\text{ times }} \times \{0\}^{\omega} \big) \right] .
\end{align*}
Then for each $N\in \N$ we have that 
\begin{align*}
A \cap \Phi_{2N+1}\supset 
 \big(
\underbrace{F_{2N + 1} \times \dots \times F_{2N + 1}}_{2N +1
\text{ times }} \times \{0\}^{\omega} \big)
\setminus 
\big(\underbrace{F_{2N} \times \dots \times F_{2N}}_{2N +1
\text{ times }} \times \{0\}^{\omega} \big),
\end{align*}
so $|A \cap \Phi_{2N+1}|\geq |F_{2N+1}|^{2N+1} - |F_{2N}|^{2N+1}=
2^{(2N+1)^2} - 2^{2N(2N+1)}$, which in turn implies that 
\begin{equation*}
\lim_{N \to \infty}\frac{|A \cap \Phi_{2N+1}|}{|\Phi_{2N+1}|} \geq  \lim_{N \to \infty}
\frac{2^{(2N+1)^2} - 2^{2N(2N+1)}}{2^{(2N+1)^2}}
=1-\lim_{N \to \infty} \frac{1}{2^{2N+1}} = 1 
\end{equation*}
Thus, we have that $\overline{\diff}_{\Phi}(A)=1$. 
We are left with proving there is no infinite set $B\subset G^{\omega}$ so that 
$B + B \subset A$. 

We denote by $\boldsymbol{0}$ the identity element $(0,0, \dots, 0, \dots)$ of $G^{\omega}$.
Consider the map $\theta \colon  G \to \N_0$ defined in \eqref{eq degree function}. Similarly to \eqref{eq norm index}, let 
$w \colon G^{\omega} \to \N_0$ be the function given by $w(\boldsymbol{g})=\max\{\theta(g_i) \colon i\in \N\} $. We also define $\tau \colon G^{\omega} \to \N_0$ by 
$\tau(\boldsymbol{0})= 0 $ and
$\tau(\boldsymbol{g})= \max\{i\in \N: g_i \neq 0\}  $ for $\boldsymbol{g}\neq 
\boldsymbol{0}$. As in the finite dimensional case, for 
$\boldsymbol{g}, \boldsymbol{g}' \in G^{\omega}, 
w(\boldsymbol{g}+ \boldsymbol{g}' )\leq \max\{w(\boldsymbol{g}),
w(\boldsymbol{g}') \}$ and if 
$w(\boldsymbol{g}) \neq w(\boldsymbol{g}')$, then 
$w(\boldsymbol{g}+ \boldsymbol{g}' )= \max\{w(\boldsymbol{g}),
w(\boldsymbol{g}') \}$. Observe that 
\begin{equation}\label{eq_inf_2}
    A=\{\boldsymbol{g} \in G^{\omega} \mid
    w(\boldsymbol{g}) \text{ odd }, w(\boldsymbol{g})\geq 3,
    \tau(\boldsymbol{g})\leq w(\boldsymbol{g}) \} = \bigsqcup_{m\in \N} \{\boldsymbol{g} \in G^{\omega} \mid
    w(\boldsymbol{g}) = 2m+1,\tau(\boldsymbol{g}) \leq 2m+1\}.
\end{equation}
Assume there is an infinite set $B \subset G^{\omega}$ so that 
$B+B \subset A$, and without loss of generality that 
$\boldsymbol{0} \notin B$. 

First, suppose that the set $\{w(\boldsymbol{b}): \boldsymbol{b}\in B\}$ is 
finite, and take $M\in \N$ so that $w(\boldsymbol{b})\leq 2M+1$ for all 
$\boldsymbol{b} \in B$. Then for all $\boldsymbol{b},
\boldsymbol{b}'\in B$ we have that 
$w(\boldsymbol{b}+\boldsymbol{b}')\leq \max\{ w(\boldsymbol{b}),
w(\boldsymbol{b}')\}\leq 2M+1$. 
Therefore, using \eqref{eq_inf_2} we get
\begin{equation}\label{eq_inf}
B + B \subset 
\bigsqcup_{1\leq m \leq M} \{\boldsymbol{g} \in G^{\omega} \colon
w(\boldsymbol{g}) = 2m+1 \text{ and } \tau(\boldsymbol{g}) \leq 2m+1\} \subset \Phi_{2M+1},
\end{equation}
which implies that $B+B$ is finite and in particular $2B$ is finite. Since $B$ is 
infinite, by pigeonhole principle, there exists $\mathbf{a} \in \Phi_{2M+1}$ and an infinite 
subset $B' \subset B$ such that $2\boldsymbol b = \mathbf{a}$ for all $\boldsymbol b \in B'$. 
Similarly, we can suppose that the 
first $2M+1$ coordinates of the elements in $B'$ are equal. Let 
$\boldsymbol b,\boldsymbol{b'}$ be two distinct elements in $B'$. 
Since $b_i = b'_i$ for all $i \leq 2M+1$, there exists $j > 2M+1$ 
such that $b_j \neq b_j'$. But then, since $a \in \Phi_{2M+1}$, 
$a_j = 0$ and hence  $2b_j = 2b_j'  =  0$. This implies that 
$b_j, b'_j \in \{0,\frac{1}{2}\}$ and therefore, using that $b_j \neq b_j'$, we must have $b_j + b_j' = \frac{1}{2}$,
contradicting \eqref{eq_inf}. Thus, the set 
$\{w(\boldsymbol{b}): \boldsymbol{b}\in B\}$ has to be infinite.

If $\{w(\boldsymbol{b}): \boldsymbol{b}\in B\}$ is infinite, let $\boldsymbol{b}, \boldsymbol{b'} \in B$ with 
$w(\boldsymbol{b'})> w(\boldsymbol{b})>0$.
Notice that $w(2\boldsymbol{b'})=
w(\boldsymbol{b'})-1$ and therefore, since $2\boldsymbol{b'}
\in A$, from \eqref{eq_inf_2} we have that 
$w(2\boldsymbol{b'}) $ is odd, so $w(\boldsymbol{b'})$ is even. Then we have that 
$w(\boldsymbol{b}+\boldsymbol{b'})=
\max\{w(\boldsymbol{b}), w(\boldsymbol{b'}) \}=
w(\boldsymbol{b'})$ which is even, and from \eqref{eq_inf_2} this 
contradicts the fact that $\boldsymbol{b}+ \boldsymbol{b'}\in A$.

To summarize, $A$ has full upper density with respect to the q.i.d.
F\o lner sequence $\Phi$ with ratio $\alpha_{\Phi}=1$, but there is no infinite set 
$B\subset G^{\omega}$ so that $B+B \subset A$. 
\end{proof}

\subsection{The F\o lner sequence has to be quasi-invariant with respect to doubling} \label{section qid is necesary}

Here we show that, if $G$ is a countable abelian group
with $2G$ infinite (note that if $[G:2G]<\infty$, then $2G$ 
is infinite), then the q.i.d. assumption is necessary 
for a density solution to the unrestricted $B+B$ problem.
We include the following useful remark, whose proof is
straightforward, so it is omitted.

\begin{remark}\label{non_qid_remark}
    Let $G$ be a countable abelian group and $\Phi=(\Phi_N)_{N\in \N}$ be a 
    \Folner{} sequence in $G$. Then 
    \begin{enumerate}[(i)]
    \item \label{non_qid_remark i} If $\Psi=(\Psi_N)_{N\in \N}$ is a sequence of 
    subsets of $G$ so that $\Psi_N\subset \Phi_N$ for all $N$ and $|\Psi_N| / |\Phi_N| \to 1$ as $N\to \infty$, then $\Psi$ is also a \Folner{} sequence in $G$.
    \item \label{non_qid_remark ii} If $H$ is a subgroup of $G$ so that 
    $[G:H] =\infty$, then for every $g\in G$ we have $\diff _{\Phi} (g+H)=0.$
    \item \label{non_qid_remark iii} Using the sub-additivity of the density, if $E_1,E_2 \subset G$ are such that the densities
$\diff_{\Phi}(E_1), \diff_{\Phi}(E_2)$ exist and $\diff_{\Phi}(E_1)=1$, 
then the density $\diff_{\Phi}(E_1 \cap E_2)$ exists and it is equal to 
$\diff_{\Phi}(E_2)$.
\end{enumerate}
\end{remark}

To achieve our aim we need the following lemma.

\begin{lemma}\label{non_qid_lemma}
    Let $G$ be a countable abelian group with $2G$ infinite, let $\Phi$ 
    be a \Folner{} sequence 
    in $G$ that is not quasi-invariant with respect to doubling, and let 
    $G=\{x_1, x_2, x_3, \ldots \}$ be an enumeration of $G$. Then there 
    is a subsequence $\Psi=(\Psi_N)_{N\in \N}$ of $\Phi$ and a \Folner{} sequence
    $F=(F_N)_{N\in \N}$ in $G$ so that for all $N$, $F_N \subset \Psi_N$,
    $(F_j +x_i) /2 \cap F_N = \emptyset$ whenever $i,j<N$,
    $$ \lim_{N\to \infty} \frac{|F_N|}{|\Psi_N|}=1 \text{ and } 
    \lim_{N\to \infty} \frac{|F_N/2 \cap F_N|}{|F_N|}=0.$$
\end{lemma}

\begin{proof}
    Since $\Phi$ is not q.i.d., we may pass to a subsequence, which by abuse of notation we also denote by $\Phi=(\Phi_N)_{N\in \N}$, so that 
\begin{equation}
\lim_{N\to \infty} \frac{|\Phi_N/2 \cap \Phi_N|}{|\Phi_N|} =0.    
\end{equation}

From the first isomorphism theorem for groups we have that 
$G/\ker (D) \cong 2G $, and since $2G$ is infinite, we have 
that $[G: \ker (D)]=\infty$. Then from Remark \ref{non_qid_remark} \eqref{non_qid_remark ii}, each coset of $\ker(D)$ in $G$ has zero density 
with respect to $\Phi$. We will inductively construct a strictly increasing 
sequence of natural numbers $(N_k)_{k\in \N}$ and a sequence of sets 
$F=(F_k)_{k\in \N}$ so that, writing $\Psi_k=\Phi_{N_k}$, we have for all 
$k\in \N$, $F_k \subset \Psi_k$, $\frac{|F_k|}{|\Psi_k|}>1-\frac{1}{k}$,
and for all $1\leq i, j<k$, $(F_j+x_i) /2 \cap F_k =\emptyset$.

Let $N_1=1$ and take $F_1=\Phi_1$.
Now, assume that for some $k\geq 1$ we have constructed $N_1<N_2<\ldots< N_k$ and 
$F_1, \ldots, F_k$ so that the previous hold. 

Observe that $\bigcup_{i,j=1} ^{k} (F_j + x_i)/2$ is a (possibly empty) 
finite union of cosets of $\ker(D)$, so 
\begin{equation*}
    \diff_{\Phi} \bigg(\bigcup_{i,j=1} ^{k} (F_j + x_i)/2\bigg)=0,
\end{equation*}
and therefore there is $N_{k+1}>N_k$ so that 
\begin{equation}
    \frac{|\Phi_{N_{k+1}}\setminus \bigcup_{i,j=1} ^{k} 
    (F_j + x_i)/2 |}{
    |\Phi_{N_{k+1}}|} >1-\frac{1}{k+1}.
\end{equation}
Taking 
$F_{k+1}=\Phi_{N_{k+1}}\setminus \bigcup_{i,j=1} ^{k} (F_j + x_i)/2$ 
one sees that $F_1, \ldots, F_{k+1}$ have the desired properties. 

Observe that $|F_k | / |\Psi_k|\to 1$ as $k\to \infty$, so from 
Remark \ref{non_qid_remark} \eqref{non_qid_remark i} we 
have that 
$F$ is indeed a \Folner{} sequence in $G$. From the construction we get that 
$(F_j + x_i) /2 \cap F_k=\emptyset $ whenever $i,j<k$. Finally, using that 
$F_k\subset \Psi_k$, and that
$\frac{|\Psi_k/2 \cap \Psi_k|}{|\Psi_k|}= 
\frac{|\Phi_{N_k}/2 \cap \Phi_{N_k}|}{|\Phi_{N_k}|} \to 0$
and $|\Psi_k|/ |F_k| \to 1$ as $k\to \infty$, we get that 
$\frac{|F_k /2 \cap F_k|}{|F_k|} \to 0$ as $k\to \infty$. This 
concludes the proof of the lemma.
\end{proof}

We are now ready to prove the necessity of the q.i.d. assumption. 

\begin{proposition}\label{quasi-invariant to doubling is necessary general}
Let $G$ be a countable abelian group so that $2G$ is infinite, 
and let $\Phi$ be a \Folner{} sequence in 
$G$ that is not quasi-invariant with respect to
doubling. Then there exists a set $A\subset G$ with 
$\overline{\diff}_{\Phi}(A)=1$ such that $t+B+B \not\subset A$ for any infinite set $B\subset G$ and any element $t\in G$. 
\end{proposition}

\begin{proof}
Let us fix an enumeration of $G$ and write $G=\{x_1,x_2,x_3,\ldots\}$. Since
$\Phi$ is not quasi-invariant with respect to doubling, we may use Lemma \ref{non_qid_lemma} to find a subsequence $\Psi=(\Psi_N)_{N\in \N}$ of 
$\Phi$ and a \Folner{} sequence $F=(F_N)_{N\in \N}$ in $G$ so that 
for all $N$, $F_N \subset \Psi_N$,
$(F_j +x_i) /2 \cap F_N = \emptyset$ whenever $i,j<N$,
\begin{equation}\label{non_qid_prop_eq_1}
\lim_{N\to \infty} \frac{|F_N|}{|\Psi_N|}=1 \text{ and } 
\lim_{N\to \infty} \frac{|F_N/2 \cap F_N|}{|F_N|}=0.
\end{equation}

We will now choose a subsequence 
$(F_{N_k})_{k\in \N}$ of $(F_N)_{N\in \N}$ as follows. For $k\in \N$, let 
$c(k)$ be a small positive constant to be determined later. Now, given 
$F_{N_1},\ldots,F_{N_{k-1}}$, we choose $N_k > N_{k-1}$  
large enough so that
\begin{enumerate}[(i)]
    \item \label{6.1 2} $|F_{N_k} \cap F_{N_k}/2|<c(k) |F_{N_k}|$ 
    \item \label{6.1 4} $\left| F_{N_k} \cap \bigcup_{i,j=1}^{k-1} 2F_{N_j} + x_i  \right| < c(k) |F_{N_k}|$
    \item \label{6.1 5} $\left|  \bigcup_{i=1}^{k-1}
    F_{N_k} \triangle (F_{N_k} + x_i)\right|< c(k) |F_{N_k}|$.
\end{enumerate}
Let us comment on why such a choice of $N_k$ is possible. 
For \eqref{6.1 2}, it suffices to use the second equation in 
\eqref{non_qid_prop_eq_1} and 
\eqref{6.1 4} is possible because 
$\bigcup_{i,j=1}^{k-1} 2F_{N_j} + x_i$ is a finite set, while 
$|F_N|\to \infty$ as $N\to \infty$. Finally, for \eqref{6.1 5}, one simply has 
to use the fact that $F$ is a \Folner{} sequence.
With that in mind, we choose $c(k)$ so that, if 
$$A_{N_k}: = \left( F_{N_k} \cap \bigcap_{i=1}^{k-1} F_{N_k} + x_i \right)\setminus \left( F_{N_k}/2
\cup \bigcup_{i,j=1}^{k-1} 2F_{N_j} + x_i \right),$$
then $|A_{N_k}|>\left(1-\frac{1}{k}\right) |F_{N_k}|$. 
Letting $A=\bigcup_{k\in \N} A_{N_k}$ it follows by the latter
that $\lim_{k\to \infty} \frac{|A\cap F_{N_k}|}{|F_{N_k}|}=1$.
Since $F_{N_k}\subset \Psi_{N_k}$ and $\frac{|F_{N_k}|}{|\Psi_{N_k}|}\to 1$
as $k\to \infty$, the previous implies that $\lim_{k\to \infty} 
\frac{|A\cap \Psi_{N_k}|}{|\Psi_{N_k}|}=1$, and since $\Psi$ is 
a subsequence of $\Phi$, this implies that $\overline{\diff}_{\Phi}(A)=1$. 

It remains to prove that for any infinite $B\subset G$ and any $t\in G$, 
$t+B+B \not \subset A$. Assume that there is some 
$t\in G$ and some infinite $B\subset G$ so that $t+B+B\subset A$. 
Let $b_1 \in B$ so that $t+2b_1 \in A_{N_{k_1}}$ for some $k_1 \in \N$. 
Let $i_1, i_2 \in \N$ so that $x_{i_1}=t+2b_1$, $x_{i_2}=-(t+2b_1)$. Since 
the $A_{N_k}$'s are finite, there is $b_2\in B$ so that $t+b_1 + b_2 
\in A_{N_{k_2}}$ for some $k_2 > \max\{i_1, i_2, k_1\}$. We have that 
$t+2b_2 \in A_{N_{k_3}}$ for some $k_3 \in \N$. 

If $k_3 > k_2$, then from the definition of $A_{N_{k_3}}$ we have that 
$t+2b_2 \notin 2 F_{N_{k_2}} + x_{i_2}$. On the other hand, since 
$t+b_1 + b_2 \in A_{N_{k_2}}\subset F_{N_{k_2}},$ 
we have that $2t + 2b_1 + 2b_2 \in 2 F_{N_{k_2}}$ which implies that 
$t+2b_2 \in 2F_{N_{k_2}} + x_{i_2}$, so we reach a contradiction. 

If $k_3 =k_2$, then we have that $t+b_1 + b_2\notin F_{N_{k_2}}/2$, so 
$2t + 2b_1 + 2b_2 \notin F_{N_{k_2}}$. Since $k_2> i_2$, $t+2b_2 \in 
A_{N_{k_2}}$ implies that $t+2b_2 \in F_{N_{k_2}} + x_{i_2}$, which in turn 
gives that $2t + 2b_1 + 2b_2 \in F_{N_{k_2}}$, so again we reach a 
contradiction. 

Finally, if $k_3 < k_2$, then $t+2b_2\in A_{N_{k_3}}$ implies that 
$t+2b_2\in F_{N_{k_3}}$, so $2t+2b_1+2b_2\in F_{N_{k_3}} + x_{i_1}$, and 
therefore $t+b_1 +b_2\in (F_{N_{k_3}} + x_{i_1})/2$. Since $k_3, i_1 < k_2$,
we have that $F_{N_{k_2}}\cap (F_{N_{k_3}} + x_{i_1})/2=\emptyset$, 
and therefore $t+b_1 + b_2\notin F_{N_{k_2}}$, so also 
$t+b_1 + b_2\notin A_{N_{k_2}}$, which again is a contradiction. Hence
in every case we reach a contradiction, so after all, for any infinite 
$B\subset G$ and any $t\in G$, $t+B+B \not \subset A$.
\end{proof}

We also show that non-q.i.d \Folner{} sequences always exist in countable abelian groups.

\begin{lemma} \label{existence of non qid}
Let $G$ be a countable abelian group. Then, there exist \Folner{} sequences in $G$ which are not quasi-invariant with respect to doubling.
\end{lemma}

\begin{proof}
Let $\Phi=(\Phi_N)_{N\in \N}$ be a \Folner{} 
sequence in $G$. Then, as each $\Phi_N$ is a finite 
set we can find $g_N\in G$ such that $g_N \notin \Phi_N-2\Phi_N$. Consider the sequence $\Psi=(\Psi_N)_{N\in \N}$ defined via $\Psi_N=g_N+\Phi_N$ for every $N\in \N$. It follows by lemma \ref{lemma aux folner 3} \eqref{shifts of Folner} (see also Remark \ref{remark about shifts of Folner}) that $\Psi$ is a \Folner{} sequence in $G$. Moreover, for any $N\in \N$ we see by the choice of $g_N$ that $\Psi_N/2 \cap \Psi_N = \emptyset$. Indeed, if this wasn't the case we would have $2(\Psi_N/2) \cap (2\Psi_N) \neq \emptyset$ for some $N\in \N$ and this in turn would imply that $(g_N+\Phi_N) \cap (2g_N+2\Phi_N) \neq \emptyset$. But this contradicts the fact that $g_N \notin \Phi_N-2\Phi_N$ and thus we conclude. 
\end{proof}

As an immediate consequence of the construction in 
Proposition \ref{quasi-invariant to doubling is necessary general}, combined with the fact that non-q.i.d. \Folner{} sequences always exist, we deduce the following. 

\begin{corollary} \label{no B+B in full density always}
Let $G$ be a countable abelian group so that $2G$ is infinite. 
Then, there exists a set $A\subset G$ with upper Banach density equal to $1$ and 
such that $t+B+B \not\subset A$ for any infinite set $B\subset G$ and any 
element $t\in G$.    
\end{corollary}

\subsection{The case of infinite index $[G:2G]=\infty$ }\label{inf_index}

As a consequence of \cite{ackelsberg2024counterexamples}, if 
$[G:2G]=\infty$, then for any $\varepsilon>0$ there exists a set with 
upper Banach density at least $1-\varepsilon$ that contains no shift of 
an infinite sumset. 
However, it could be that sets of full density 
always contain $t+B+B$. An interesting dichotomy manifests itself in 
this case. Indeed, as we show below, if $2G$ is a finite set (e.g., 
when $G=\F_2 ^{\omega}$), then any set of full upper Banach density 
contains a shifted sumset $t+B+B$. On the other hand, once $2G$ is an 
infinite set (assuming $[G:2G]=\infty$) we construct -- along any 
given \Folner{} sequence -- a set of full density that fails to 
contain such sumsets.

We begin with a simple observation.

\begin{lemma}
Let $G$ be an abelian group with $[G:2G]=\infty$ and let 
$\Phi$ be any \Folner{} sequence in $G$. 
Then there is a set $A\subset G$ with ${\diff}_{\Phi}(A)=1$ 
so that for all infinite $B\subset G$, $B+B \not \subset A$. 
\end{lemma}

\begin{proof}
From Remark \ref{non_qid_remark} \eqref{non_qid_remark ii} we have that 
$\diff_{\Phi}(2G)=0$. Let $A=G\setminus 2G$. Then $\diff_{\Phi}(A)=1$
and $A\cap 2G=\emptyset$, so for each nonempty $B\subset G$, 
$2B\cap A=\emptyset$. This concludes the proof. 
\end{proof}

Allowing for the possibility of shifted sumsets 
makes the situation more delicate. 

\begin{proposition} \label{t+B+B in finite 2G}
Let $G$ be a countable abelian group so that $2G$ is finite. If $A\subset G$ has upper Banach density $1$, then for every $t\in A$ there is some infinite $B\subset A$ so that $t+B+B\subset A$.
\end{proposition}

\begin{proof}
Let $\Psi=(\Psi_N)_{N\in \N}$ be a \Folner{} sequence along which the 
density of $A$ is equal to $1$, i.e.
$$\lim_{N\to \infty} \frac{|A\cap\Psi_N|}{|\Psi_N|}=1.$$
From the first isomorphism theorem for groups we have that 
$G/ \ker(D)\cong 2G$, which implies that $s:=[G: \ker(D)]<\infty.$ Then, using \cite[Lemma 5.4]{charamaras_mountakis2024} we have that 
$\diff_{\Psi}(\ker(D))=\frac{1}{s}>0,$ where the previous density exists. 

Let $t\in A$. Then $\diff_{\Psi}(A)=1$, so from Remark \ref{non_qid_remark} \eqref{non_qid_remark iii} 
we have that $\diff_{\Psi}(A\cap \ker(D))=\frac{1}{s}.$
Pick $b_1 \in A\cap \ker(D)$. Since $b_1 \in \ker(D)$, we have 
$t+2b_1 = t \in A$.

Since $\diff_{\Psi}(A-t-b_1)=1$, we 
have that $\diff_{\Psi}(A\cap (A-t-b_1) \cap \ker(D))=\frac{1}{s},$
so we may pick $b_2 \in A\cap (A-t-b_1) \cap \ker(D)$, with $b_2 \neq 
b_1$. Again, 
since $b_2 \in \ker(D)$, we have $t+2b_2 = t \in A$, and since 
$b_2 \in A-t-b_1$ we have $t+b_1 +b_2 \in A$. 

Next, since $\diff_{\Psi}(A-t-b_2)=1$, we have 
$\diff_{\Psi}(A\cap (A-t-b_1) \cap (A-t-b_2)\cap \ker(D))=\frac{1}{s},$
so we may pick $b_3 \in A\cap (A-t-b_1) \cap (A-t-b_2)\cap \ker(D)$, with 
$b_3 \neq b_1, b_2$. 
Again, since $b_3 \in \ker(D)$, we have $t+2b_3 = t \in A$, and since 
$b_3 \in (A-t-b_1) \cap (A-t-b_2)$ we have $t+b_1 +b_2, t+b_1 +b_3 \in A$. 

Continuing inductively, we end up finding a sequence of different elements 
$(b_j)_{j\in \N}\subset A$ so that for all $i,j \in \N$ with $i\neq j$ we 
have $t+b_i + b_j \in A$ and $t+2b_i =t \in A$. Hence for the infinite
set $B=\{b_j: j\in \N\}\subset A$ we have $t+B+B\subset A$.
\end{proof}

\begin{remark*}
Recall that a set $A\subset G$ is thick if for any finite set 
$F\subset G$ there is some $t\in G$ such that $t+F\subset A$. It is 
not difficult to see that this happens if and only if for any finite 
set $F\subset G$ there is some $t\in A$ such that $t+F\subset A$. 
Now, it is an easy exercise to verify that sets of 
upper Banach density $1$ are thick. Also, it is well-known (see for 
example \cite[Lemma $4.5$]{Berg-Hindman}) that thick sets are $IP$-
sets. That is, if $A\subset G$ is thick, there exists an infinite set 
$B\subset G$ such that $FS(B):=\big\{\sum_{b\in H}b: H\subset B,\ 
\text{$H$ is finite} \big\} \subset A$. 

Therefore, in the setting of Proposition \ref{t+B+B in finite 2G}, the set $A$ is thick and so there exists $t\in A$ so that $t+2G \subset A$. Moreover, $A-t$ is also thick and thus we may find $B\subset A-t$ infinite with $FS(B)+t \subset A$. In particular, since $2B\subset 2G \subset A-t$, we actually have that $(t+B+B) \cup (t+FS(B)) \subset A$.  
\end{remark*}

We conclude with the following strengthening of Proposition \ref{quasi-invariant to doubling is necessary general} in the case that $[G:2G]=\infty$. 

\begin{proposition}
    Let $G$ be an abelian group so that $2G$ is infinite and $[G:2G]=\infty$,
    and let $\Phi=(\Phi_N)_{N\in \N}$ be any
    \Folner{} sequence in $G$. Then there is a set 
    $A\subset G$ so that $\overline{\diff}_{\Phi}(A)=1$ and for all 
    infinite $B\subset G$ and $t\in G$, 
    $t+B+B \not \subset A$. 
\end{proposition}

\begin{proof}
From the first isomorphism theorem for groups we have that $G/\ker(D)\cong 2G$,
so in particular $[G:\ker(D)]=\infty$, and therefore from 
Remark \ref{non_qid_remark} \eqref{non_qid_remark ii} we infer that 
$\diff_{\Phi}(g + \ker(D))=0$ for all $g\in G$.
Also, since $[G:2G]=\infty$, again from Remark \ref{non_qid_remark} 
\eqref{non_qid_remark ii} we have $\diff_{\Phi}(g+2G)=0$ for all $g\in G$.

We inductively construct
a strictly increasing sequence $(N_k)_{k\in \N}$ of natural numbers and a 
sequence of pairwise disjoint finite sets $A_k\subset \Phi_{N_k}$, $k\in \N$, 
as follows. 
Take $N_1=1$ and $A_1=\Phi_1$.
Now, for $k\geq 2$, given $N_1, \ldots, N_{k-1}$ and $A_1, \ldots, A_{k-1}$,
let $D_k=\bigcup_{i=1} ^{k-1} A_i$, which is finite. Then $E_{1,k}=D_k + 2G$ is
a finite union of cosets of $2G$, so $\diff_{\Phi}(E_{1,k})=0$, 
and $E_{2,k}=D_k + \ker(D)$ is a finite union of cosets of $\ker(D)$,
so also $\diff_{\Phi}(E_{2,k})=0$. Hence, letting $E_k=E_{1,k}\cup E_{2,k}$
we have $\diff_{\Phi}(E_k)=0$, so there is $N_k > N_{k-1}$ with
\begin{equation}\label{non_qid_eq_sec_1}
    \frac{|\Phi_{N_k}\setminus E_k|}{|\Phi_{N_k}|}>1-\frac{1}{k}.
\end{equation}
Take $A_k =  \Phi_{N_k} \setminus E_k$, and let
$A=\bigcup_{k=1}^{\infty} A_k$. Then from \eqref{non_qid_eq_sec_1}
we have that 
$\lim_{k\to \infty} \frac{|A\cap \Phi_{N_k}|}{|\Phi_{N_k}|}=1$, so in 
particular $\overline{\diff}_{\Phi}(A)=1$.
Also, since each $\Phi_N$ is finite, from the construction 
one can see that for all $g\in G$, the sets $A\cap (g+2G)$ and 
$A\cap(g+ \ker(D))$ are finite. 

Now, assume that there is some $t\in G$ and some infinite $B\subset G$ so that 
$t+B+B \subset A$. Suppose that some coset of $\ker(D)$ contains infinitely
many elements of $B$, i.e. there is $g_0 \in G$ so that 
$B'=B\cap (g_0 + \ker(D))$ is infinite. It follows that $t+B' +B' \subset A \cap (t+ 2g_0 + \ker(D))$, and therefore $A\cap (t+ 2g_0 + \ker(D))$
is infinite, contradicting the fact that $A$ has finite intersection with 
every coset of $\ker(D)$. 

Thus, for every $g \in G$, $B\cap (g + \ker(D))$ is finite, and 
since $B$ is infinite, there is an infinite $B' \subset B$ so that any two elements of $B'$ belong to different cosets of $\ker(D)$. In 
particular for $b_1, b_2 \in B'$ with $b_1 \neq b_2$, we have 
$2b_1 \neq 2b_2$, so that $\{t+2b' : b' \in B'\}$ is infinite. But 
$\{t+2b' : b' \in B'\}\subset A$ and $\{t+2b' : b' \in B'\}\subset t+2G$, and 
so $A\cap (t+ 2G)$
is infinite, contradicting the fact that $A$ has finite intersection with 
every coset of $2G$. Hence after all, if $t\in G$ and $B\subset G$ is 
infinite, then $t+B+B \not \subset A$, which concludes the proof. 
\end{proof}

\section{Existence of \Folner{} sequences that are quasi-invariant with respect to doubling}
\label{section_existence}
The goal of this section is to prove that any abelian group $G$ with $[G:2G]<\infty$ and $|\ker(D)|<\infty$ admits a q.i.d. \Folner{} sequence. In fact, we establish the following even stronger result. Recall that 
$\alpha_G=\sup\{\alpha_{\Phi} \colon \Phi \in \mathcal{F}_G\}$, 
where $\mathcal{F}_G$ denotes the collection of all \Folner{} sequences
in $G$.

\begin{theorem}\label{theorem bound for quasi-invariant}
    If $G$ is a countable abelian group with $\ell = [G:2G]<\infty$ 
    and $r = |\ker(D)|<\infty$, then the following hold:
    \begin{enumerate}
        \item[\namedlabel{item bound for}{(a)}] $\alpha_G= \min\{1,\frac{r}{\ell}\}$. In particular, 
        $G$ admits \Folner{} sequence that is quasi-invariant with 
        respect to doubling.
        \item[\namedlabel{sup_is_attained}{(b)}] There exists a \Folner{} sequence $\Phi$ in $G$ such that $\alpha_\Phi = \alpha_G$.
    \end{enumerate}
\end{theorem}

We start by proving \ref{sup_is_attained} of Theorem \ref{theorem bound for quasi-invariant}.

\begin{proof}[Proof of \ref{sup_is_attained} in Theorem \ref{theorem bound for quasi-invariant}]

By definition, there is a sequence 
$\alpha_k\to \alpha_G$
and a sequence $(\Psi^{(k)})_{k\in \N}=\big(\big(\Psi^{(k)}_N\big)_{N\in \N}\big)_{k\in \N}$ of \Folner{}
sequences in $G$ so that for all $k$, $\alpha_k=\alpha_{\Psi^{(k)}}$. 
Let $G=\{x_1, x_2, x_3,\ldots\}$ be an enumeration of the elements of $G$. Then for each 
$k\in \N$ there is $N_k \in \N$ so that for all $N\geq N_k$ we have 
\begin{equation}\label{sup_eq_1}
\frac{\big| \Psi^{(k)}_N \triangle \big( x_s+ \Psi^{(k)}_N \big) \big|}{\big| \Psi^{(k)}_N \big|}< \frac{1}{k},\ \text{for all}\ s\in \{1,\ldots,k\}, 
\:\text{ and } \:
\frac{\big|\Psi^{(k)}_N \cap (\Psi^{(k)}_N/2)\big|}{\big|\Psi^{(k)}_N\big|}
> \alpha_k - \frac{1}{k}.
\end{equation}

Take a strictly increasing sequence $(N_k)_{k\in \N}$ so 
that for all $k\in \N$, \eqref{sup_eq_1} holds for $N_k$ and for each $k$, let 
$\Phi_k=\Psi^{(k)}_{N_k}$. For $x=x_{n_0} \in G$ and $k\geq n_0$ we have 
$\frac{|\Phi_k \triangle (x+\Phi_k)|}{|\Phi_k|}<\frac{1}{k}\xrightarrow{k\to \infty} 0$, so $\Phi=(\Phi_k)_{k\in \N}$ is indeed a \Folner{} sequence in $G$.
Also, $\frac{\left|\Phi_k \cap (\Phi_k/2)\right|}{|\Phi_k|}> 
\alpha_k - \frac{1}{k} \xrightarrow{k\to \infty} \alpha_{G}$, so $\alpha_{\Phi}
\geq \alpha_G$. By definition, we also have that 
$\alpha_{\Phi} \leq \alpha_G$, so after all, $\alpha_{\Phi} = \alpha_G$. This
concludes the proof of the lemma.
\end{proof}

Now we move to the proof of \ref{item bound for} of
Theorem \ref{theorem bound for quasi-invariant}. For 
that we need the following lemma.

\begin{lemma}\label{final_useful_lemma}
Let $G$ be a countable abelian group with $\ell=[G : 2G] < \infty$ and 
$r=|\ker(D)|<\infty$. Let $g_1, \ldots, g_{\ell}\in G$ so that $g_1 =0$
and $G=\bigsqcup_{i=1}^{\ell} g_i + 2G$. If $\Psi=(\Psi_N)_{N\in \N}$   
is any \Folner{} sequence in $G$, and we let $F_N= \bigsqcup_{i=1}^{\ell}
(g_i+2\Psi_N)$, $N\in \N$, then $F=(F_N)_{N\in \N}$ is also a \Folner{} sequence in $G$.
\end{lemma}

\begin{proof}
Let $g\in G$ and $\varepsilon>0$.
Then there exist $1\leq i_0\leq \ell$ and $h\in G$ such that $g=g_{i_0}+2h$. Then 
\begin{equation}\label{folner eq1}
    (g+F_N)\cap F_N
    = \bigg(\bigsqcup_{i=1}^{\ell} (g_{i_0}+g_i+2(h+\Psi_N))\bigg)
    \cap\bigg(\bigsqcup_{i=1}^{\ell} (g_j+2\Psi_N)\bigg).
\end{equation}
Since the cosets $g_j+2G$ are disjoint, it follows 
that for each $1\leq i\leq \ell$,
there exists a unique $1\leq j(i)\leq \ell$ such that
$g_{i_0} + g_i \in g_{j(i)} + 2G$, so there is
$y_i \in G$ such that $g_{i_0} + g_i = g_{j(i)} + 2 y_i$.
In addition, if for some $i_1 \neq i_2$ we have 
$j(i_1) = j(i_2)$, then we have that $g_{i_1}- g_{i_2} = 
2y_{i_1} - 2y_{i_2} \in 2G$, so $g_{i_1} + 2G = g_{i_2} + 2G$, 
which is a contradiction. Therefore, the map
$j: \{1, \dots, \ell \} \to \{1, \dots, \ell \}, 
i \mapsto j(i)$ is a bijection,
and then \eqref{folner eq1} becomes
\begin{equation*}
    (g+F_N)\cap F_N
    = \bigsqcup_{i=1}^{\ell}\big((g_{j(i)}+2(y_i+h+\Psi_N))
    \cap(g_{j(i)}+2\Psi_N)\big)
    = \bigsqcup_{i=1}^{\ell} \big(g_{j(i)} + (2(y_i+h+\Psi_N)
    \cap 2\Psi_N)).
\end{equation*}
We thus have that
\begin{equation}\label{folner eq2}
    |(g+F_N)\cap F_N|
    = \sum_{i=1}^{\ell} |2(y_i+h+\Psi_N)\cap 2\Psi_N)|
    \geq \sum_{i=1}^{\ell} |2((y_i+h+\Psi_N )\cap \Psi_N)|.
\end{equation}
Now since $\Psi$ is a \Folner{} sequence in $G$, for $N$ sufficiently large,
we have that for every $0\leq i\leq \ell $,
$$|(y_i+h+\Psi_N)\triangle\Psi_N|\leq\frac{\varepsilon}{r}|\Psi_N|,$$
and then we have that
\begin{equation}\label{folner eq3}
\frac{|2((y_i+h+\Psi_N)\cap \Psi_N)|}{|2\Psi_N|}
\geq 
1 - \frac{|2(\Psi_N \triangle (y_i+h+\Psi_N))|}{|2\Psi_N|}
\geq 
1 - \frac{r|\Psi_N \triangle (y_i+h+\Psi_N)|}{|\Psi_N|}
\geq 1 - \varepsilon.
\end{equation}
Then, combining \eqref{folner eq2} and \eqref{folner eq3} we get 
that for $N$ sufficiently large,
\begin{equation*}
    \frac{|(g+F_N)\cap F_N|}{|F_N|}
    \geq \frac{\ell(1-\varepsilon)|2\Psi_N|}{|F_N|}
    = 1-\varepsilon.
\end{equation*}
Since $\varepsilon>0$ was arbitrary, it follows that 
$\lim_{N\to\infty}\frac{|(g+F_N)\cap F_N|}{|F_N|}=1$ and so $(F_N)_{N\in \N}$ is a \Folner{} sequence.
\end{proof}

\begin{proof}[Proof of \ref{item bound for} in Theorem \ref{theorem bound for quasi-invariant}]
Fix a countable abelian group $G$ with 
$\ell=[G:2G]<\infty$ and $r=|\ker(D)|<\infty$. We need to prove that 
$\alpha_{G}= \min\{1, r/\ell\}$.

Let $\Phi=(\Phi_N)_{N\in \N}$ be a \Folner{} sequence in $G$. By 
    definition, $\alpha_{\Phi} \leq 1$. On the other hand, from 
    Lemma \ref{lemma aux folner 2} we know that 
    $\frac{|\Phi_N/2|}{|\Phi_N|}\to \frac{r}{\ell}$ as $N\to \infty$.
    For each $N$, $\frac{|\Phi_N/2\cap \Phi_N|}{|\Phi_N|}
    \leq \frac{|\Phi_N/2|}{|\Phi_N|}$, so taking $\liminf_{N \to \infty}$ we 
    obtain that $\alpha_{\Phi}\leq \frac{r}{\ell}$. Therefore, for each 
    $\Phi \in \mathcal{F}_G$, $\alpha_{\Phi} \leq \min \big\{1,\frac{r}{\ell} \big\}$, which implies that $\alpha_{G}\leq \min\{1, r/\ell\}$.

To conclude the proof, it 
suffices to prove that $\alpha_{G}\geq \min\{1, r/\ell\}$. We split the proof
into three cases, according to whether $\ell > r$, $\ell < r$ or 
$\ell=r$. 

\underline{The case $\ell > r$:}
In this case $\min\{1, r/\ell\}=r/\ell$. 
Let ${\Psi}=({\Psi}_N)_{N\in \N}$ be any \Folner{} sequence in $G$.
Let also $g_1, \ldots, g_{\ell}\in G$ so that $g_1 =0$ and 
$G=\bigsqcup_{i=1}^{\ell} g_i + 2G$. For each $N\in \N$, let  
\begin{equation*}
{H}_N ^{(1)} = \bigsqcup_{i=1}^{\ell} 
(g_i+2\Psi_N) \text{ and }  
{H}_N ^{(j)} = \bigsqcup_{i=1}^{\ell} \big(g_i+2H_N ^{(j-1)}\big),
\text{ for } j >1.
\end{equation*} 
From Lemma \ref{final_useful_lemma}, we inductively get that for all 
$j\in \N$, $H^{(j)}= \big(H^{(j)}_N\big)_{N\in \N}$
is a \Folner{} sequence in $G$. 
For each $k,N\in \N$, let $$F^{(k)}_N= \bigcup_{j=1} ^{k} H_{N} ^{(j)}.$$
From Lemma \ref{lemma aux folner 3} \eqref{item_u_f} we have that for all 
$k\in \N$, $F^{(k)}=\big(F_{N} ^{(k)}\big)_{N \in \N}$ is also a \Folner{}
sequence in $G$. We have $H^{(1)}_N/2=(2\Psi_N) /2 =\Psi_N + \ker(D)$
and for each $j>1$, $H^{(j)}_N/2 = \big(2H^{(j-1)}_N\big) /2
=H_N^{(j-1)} + \ker(D)$.
In addition, using Lemma \ref{lemma aux folner} we get 
that for each 
$j\in \N$,
$|H_N ^{(j)}|=\ell |2H_N ^{(j-1)}|=
\frac{\ell}{r} |H_N ^{(j-1)}|+ \oh_{N\to \infty} (|H_N ^{(j-1)}|)$, so
inductively we get that for each $j$, 
\begin{equation}\label{eq_w_i_o}
    \big|H_N ^{(j)}\big|= 
\Big(\frac{\ell}{r}\Big)^j |\Psi_N|+ \oh_{N\to \infty} (|\Psi_N|).
\end{equation}
For each $k,N\in \N$ we have 
$F^{(k)}_N/2 =  (\Psi_N + \ker(D))\cup
\bigcup_{j=1}^{k-1} \big(H_N ^{(j)} + \ker(D)\big)$ and
$(F^{(k)}_N/2) \setminus  F^{(k)}_N \subset (\Psi_N + \ker(D))\cup
\bigcup_{j=1}^{k-1} \big(H_N ^{(j)} + \ker(D) \big)\setminus H_N^{(j)} $. Also, from the proof of Lemma \ref{lemma aux folner} we have that 
$\frac{|\Psi_N + \ker(D)|}{|\Psi_N|}\to 1$ and 
$\frac{|H_N^{(j)} + \ker(D)|}{|H_N^{(j)}|}\to 1$ as $N\to \infty$. 
Combining those with \eqref{eq_w_i_o} we infer that 
\begin{align*}
    \frac{\big|(F^{(k)}_N/2) \setminus  F^{(k)}_N\big|}{\big|F^{(k)}_N/2\big|} &\leq 
    \frac{|\Psi_N + \ker(D)|}{\big|F^{(k)}_N/2\big|}
    + \sum_{j=1}^{k-1} \frac{\big|(H^{(j)}_N +\ker(D)) \setminus  H^{(j)}_N\big|}{\big| F_N^{(j)}\big|}\\
    &\leq 
    \frac{|\Psi_N + \ker(D)|}{|\Psi_N|} \frac{|\Psi_N|}{\big|H^{(k-1)}_N\big|}
    + \sum_{j=1}^{k-1} \Bigg(1- \frac{\big|H^{(j)}_N\big|}{\big| H^{(j)}_N+\ker(D) \big|}\Bigg)\xrightarrow{N\to\infty}
    \Big(\frac{r}{\ell}\Big)^{k-1},
\end{align*}
which implies that
$$\limsup_{N\to \infty} \frac{\big|(F^{(k)}_N/2) \setminus  F^{(k)}_N\big|}{\big|F^{(k)}_N/2\big|}
\leq \Big(\frac{r}{\ell}\Big)^{k-1}.$$ 
From Lemma \ref{lemma aux folner 2} we know that $\frac{|F_N^{(k)} /2|}{|F^{(k)}_N|}
\to \frac{r}{\ell}$ as $N\to \infty$ and therefore
$$\alpha_k := \alpha_{F^{(k)}}=
\liminf_{N \to \infty}\frac{|F_N^{(k)} \cap F_N^{(k)}/2|}{|F_N^{(k)}|} = \liminf_{N \to \infty} \frac{|F_N^{(k)}/2|}{|F_N^{(k)}|}
\left(1-\frac{|(F_N^{(k)}/2) \setminus F_N^{(k)}|}{|F_N^{(k)}/2|}  \right)
\geq \frac{r}{\ell} \left(1- \Big(\frac{r}{\ell}\Big)^{k-1}  \right).$$
Using that $(\frac{r}{\ell})^{k-1} \searrow 0$ as $k\to \infty$ we get 
$\alpha_G \geq \sup \{\alpha_k : k\in \N\}=\frac{r}{\ell}$, which 
concludes the proof in that case.

\underline{The case $\ell < r$:}
In this case we have $\min\{1, r/\ell\}=1$. 
Let $\Psi=(\Psi_N)_{N\in \N}$ be any \Folner{} sequence in $G$. 
By an application of Lemma \ref{lemma aux folner 3} \eqref{useful_item_2}, we 
have that for each $j\in \N$, 
$(\Psi_N/2^j)_{N\in \N}$ is also a \Folner{} sequence in $G$.
By the same lemma, if for each $k,N \in \N$ we define  
\begin{equation*}
F^{(k)}_N=\bigcup_{j=0}^k \Psi_N/2^j,
\end{equation*}
then we have that 
$F^{(k)}=(F^{(k)}_N)_{N\in \N}$ is also \Folner{} sequence. 
In addition, for each $j\in \N$ we have that 
$$|\Psi_N/2^j|=\Big(\frac{r}{\ell}\Big)^{j}|\Psi_N|+o_{N\to \infty}(|\Psi_N|),$$ 
by induction and Lemma \ref{lemma aux folner 2}. Therefore, we see that
$$\frac{\big|F^{(k)}_N \setminus (F^{(k)}_N/2)\big|}{\big|F^{(k)}_N\big|} 
\leq \frac{|\Psi_{N}|}{|\Psi_{N}/2^{k}|}\xrightarrow{N \to \infty} 
\Big(\frac{\ell}{r}\Big)^{k},$$
which implies that
$$\alpha_k := \alpha_{F^{(k)}}=\liminf_{N\to \infty} \frac{\big|F^{(k)}_N \cap (F^{(k)}_N/2)\big|}{|F^{(k)}_N|} = \liminf_{N \to \infty} \left(1-\frac{\big|F^{(k)}_N \setminus (F^{(k)}_N/2)\big|}{|F^{(k)}_N|}  \right) \geq 1 -\Big(\frac{\ell}{r}\Big)^{k}.$$ 
Using that $(\frac{\ell}{r})^{k} \searrow 0$ as $k\to \infty$ we get 
$\alpha_G \geq \sup \{\alpha_k : k\in \N\}=1$, which 
concludes the proof in that case.

\underline{The case $\ell=r$:}
In case $\ell=r$, we have $\min\{1, r/\ell\}=1$. 
Let $\Psi=(\Psi_N)_{N\in \N}$ be any \Folner{} sequence in $G$. As before, for each 
$k\in \N_0$ consider the \Folner{} sequence 
$F^{(k)}=\big(F^{(k)}_N\big)_{N\in \N}$ defined by 
\begin{equation*}
F^{(k)}_N=\bigcup_{j=0}^k \Psi_N/2^j.
\end{equation*}
Also, for each $k\in \N_0$, let $\alpha_k = \alpha_{F^{(k)}} = 
\liminf_{N\to \infty} \frac{|F^{(k)}_N \cap (F^{(k)}_N/2)|}{| F^{(k)}_N|}$. By employing a diagonal argument and passing to a subsequence,
we may assume that the limits exist, so 
$\alpha_k = \lim_{N\to \infty} \frac{| F^{(k)}_N \cap (F^{(k)}_N/2)|}{| F^{(k)}_N|}$.

Assume that $\sup\{ \alpha_k : k\in \N_0\}=\alpha<1$. Then for each 
$j\in \N_0$ we have 
\begin{equation}\label{in the jungle}
\big(F_N ^{(j)} /2 \big) \setminus F_N ^{(j)}=
\big(\Psi_N / 2^{j+1}\big) \setminus \big( \bigcup_{m=0} ^j \Psi_N / 2^m \big).
\end{equation}
Also, since $r=\ell$ and $F^{(j)}$ is a \Folner{} sequence, from Lemma \ref{lemma aux folner 2}
we have that $\lim_{N\to \infty} \frac{| F^{(j)}_N /2 |}{| F^{(j)}_N |}=1$. Using the previous we get that 
\begin{equation*}
\alpha_j = \lim_{N\to \infty} \frac{\big| (F^{(j)}_N /2)\cap F^{(j)}_N  \big|}{\big| F^{(j)}_N \big|}= 
\lim_{N\to \infty} \frac{\big| (F^{(j)}_N /2)\cap F^{(j)}_N  \big|}{\big| F^{(j)}_N/2 \big|}= 
\lim_{N\to \infty}\Big[1- \frac{\big| (F^{(j)}_N /2)\setminus F^{(j)}_N  \big|}{\big| F^{(j)}_N/2 \big|}\Big],
\end{equation*}
which implies that 
\begin{equation}\label{the mighty jungle}
\lim_{N\to \infty} \frac{\big| (F^{(j)}_N /2)\setminus F^{(j)}_N  \big|}{\big| F^{(j)}_N \big|} =1 - \alpha_{j}.
\end{equation}
Observe that $\Psi_N \subset F^{(j)}_N$, and combining this with
\eqref{in the jungle} and \eqref{the mighty jungle} we get that 
\begin{equation}\label{the lion sleeps tonight}
\liminf_{N\to \infty} \frac{\big| \big(\Psi_N / 2^{j+1}\big) \setminus \big( \bigcup_{m=0} ^j \Psi_N / 2^m \big)  \big|}{\big| \Psi_N \big|} \geq 1 - \alpha_{j}\geq 1-\alpha  >0.
\end{equation}
For each $k\in \N$, utilizing
\eqref{the lion sleeps tonight} for $j\in\{0, 1, \ldots, k-1\}$ we can find an 
$N_k\in \N$ such that for all $N\geq N_k$,  
\begin{equation}\label{the peaceful village}
\big| \big(\Psi_N / 2^{j+1}\big) \setminus \big( \bigcup_{m=0} ^j \Psi_N / 2^m \big)  \big|>(1-\alpha)(1-2^{-k}){\big| \Psi_N \big|},\ \text{for all}\ j\in \{0,1,\ldots,k-1\}.
\end{equation}
Observe that $F_N ^{(k)} =
\Psi_N \sqcup \bigsqcup_{j=0} ^{k-1}
\big(\Psi_N / 2^{j+1}\big) \setminus \big( \bigcup_{m=0} ^{j} \Psi_N / 2^m \big),
$ so for $N\geq N_k$ we have 
\begin{equation}\label{hakuna matata}
|F_N ^{(k)}| = |\Psi_N| + \sum_{j=0} ^{k-1} \big| \big(\Psi_N / 2^{j+1}\big) \setminus \big( \bigcup_{m=0} ^j \Psi_N / 2^m \big)  \big|
> k(1-\alpha)(1-2^{-k})|\Psi_N|.
\end{equation}
Since $F_N ^{(k)} \setminus (F_N ^{(k)}/2) \subset \Psi_N$, combining with  
\eqref{hakuna matata} we get that for $N\geq N_k$, 
\begin{equation}\label{hakuna matata 2}
\frac{\big| F_N ^{(k)} \setminus (F_N ^{(k)}/2) \big|}{\big| F_N ^{(k)}\big|}
\leq \frac{1}{k(1-\alpha)(1-2^{-k})},
\end{equation}
which implies that
$$\alpha_k := \alpha_{F^{(k)}}=\liminf_{N\to \infty} \frac{\big|F^{(k)}_N \cap (F^{(k)}_N/2)\big|}{|F^{(k)}_N|} \geq 1 - \frac{1}{k(1-\alpha)(1-2^{-k})} \xrightarrow[]{k\to \infty} 1.$$ 
This contradicts our assumption that $\sup\{ \alpha_k : k\in \N_0\}<1$. 
Therefore, $\sup\{ \alpha_k : k\in \N_0\}=1$, which implies that 
$\alpha_G \geq 1$ and concludes the proof in the case $\ell=r$, and with it 
the proof of \ref{item bound for} of Theorem \ref{theorem bound for quasi-invariant}.
\end{proof}

\appendix

\section{Properties of abelian groups and their \Folner{} sequences} 
In this appendix, we collect results regarding abelian groups and \Folner{} sequences in 
abelian groups that are used throughout the paper.  
Recall that given an abelian group $G$, a sequence  
$\Phi=(\Phi_N)_{N\in \N}$ of finite subsets of $G$ is 
a \textit{\Folner{} sequence} 
if for any $g\in G$, 
\begin{equation}\label{def_foln}
    \lim_{N\to \infty} 
\frac{|\Phi_N \cap (g+\Phi_N)|}{|\Phi_N|}=1.
\end{equation}
Let $G$ denote a countable abelian group with 
$\ell=[G:2G]< \infty$ and $r=|\ker(D)|<\infty$. We begin by proving that $\ell$ and $r$ can only be powers of $2$. 

 \begin{lemma}\label{powers_of_2}
     If $G$ is an abelian group with $\ell = [G \colon 2G]$ and $r= |\ker (D)|$ then
     \begin{itemize}
         \item if $\ell <\infty$, there exists $d_1\in \N_0$ such that $\ell = 2^{d_1}$ and
         \item if $r<\infty$, there exists $d_2\in \N_0$ such that $r= 2^{d_2}$.
     \end{itemize}
 \end{lemma}

 \begin{proof}
     Since $G$ is an abelian group, $G/2G$ and $\ker(D)$ are also abelian groups and therefore they have a $\Z$-module structure. Furthermore, $G/2G$ and $\ker(D)$ have an $\F_2 = \Z /2\Z$-module structure. Indeed, if $g \in \ker(D)$ then $2g = g + g = 0$ by definition. Likewise, if $h \in G/2G$ then $h = \gamma + 2G$ for some $\gamma \in G$ and $2h = 2 \gamma + 2G = 0 +2G$ which is the identity element. Since $\F_2$ is a field, then $G/2G$ and $\ker(D)$ are vector spaces over $\F_2$, in particular they are isomorphic to $\displaystyle \bigoplus_{i \in I_1} \F_2$ and $\displaystyle \bigoplus_{i \in I_2} \F_2$ for some index set $I_1$ and $I_2$ respectively. 

     We conclude by observing that $\ell<\infty$ (respectively $r<\infty)$ if and only if $|I_1|<\infty$ (respectively $|I_2|<\infty$) and in that case $\ell=2^{|I_1|}$ (respectively $r=2^{|I_2|}$).
 \end{proof}

\begin{lemma}\label{lemma aux folner 3}
    Let $\Phi$, $\Psi$ be \Folner{} sequences in $G$. Then the following hold:
    \begin{enumerate}[(i)]
            \item \label{item_u_f} $\Phi\cup \Psi=(\Phi_N\cup \Psi_N)_{N\in \N}$ is a \Folner{} sequence in $G$. 
    \item \label{intersection_is_Folner_2} \vspace{2mm} If $ \displaystyle 0< \liminf_{N\to \infty} \frac{|\Psi_N|}{|\Phi_N|} \leq \limsup_{N\to \infty} \frac{|\Psi_N|}{|\Phi_N|} < +\infty $ and, $ \displaystyle \liminf_{N\to \infty} \frac{|\Phi_N \cap \Psi_N|}{|\Psi_N|}>0 $ or $\displaystyle \liminf_{N\to \infty} \frac{|\Phi_N \cap \Psi_N|}{|\Phi_N|}>0$, 
    \vspace{2mm} then $\Phi \cap \Psi = (\Phi_N \cap \Psi_N)_{N\in \N}$ is a 
    \Folner{} sequence in $G$.
    \item \label{shifts of Folner} If $(g_N)_{N\in \N}$ is a sequence 
    of elements of $G$, then the sequence of shifts 
    $(g_N+\Phi_N)_{N\in \N}$ is a \Folner{} sequence 
    in $G$. 
    \item  \label{useful_item_2}  \vspace{2mm} $\Phi/2=(\Phi_N/2)_{N\in \N}$ is a \Folner{} sequence in $G$. 
    \end{enumerate}
\end{lemma}

\begin{proof} 
Statements \eqref{item_u_f} and \eqref{intersection_is_Folner_2}
follow immediately from the definitions, and while 
\eqref{shifts of Folner} should be known to aficionados, the proof is also a simple consequence of the definitions.
Let us now prove \eqref{useful_item_2}. From 
{\cite[Lemma 5.4]{charamaras_mountakis2024}} we know that $N\mapsto \wt{\Phi}_N=
\Phi_N \cap 2G$, $N\in \N$ is a \Folner{} sequence in $2G$. Observe that 
for each $N \in \N$, 
$\Phi_N/2=\wt{\Phi}_N /2$. Let $g \in G$. Then for each $N$ we have that 
$\big(g+\big(\wt{\Phi}_N/2\big) \big) \triangle \big(\wt{\Phi}_N/2\big)\subset 
\big((2g+\wt{\Phi}_N) \triangle \wt{\Phi}_N\big)/2$.

Observing that $(2g+ \wt{\Phi}_N )\triangle \wt{\Phi}_N\subset 2G$ and 
$\wt{\Phi}_N \subset 2G$, it follows from the definition of the kernel of the doubling map $D$ that 
$| \big((2g+ \wt{\Phi}_N )\triangle \wt{\Phi}_N\big)/2|=
r |(2g+ \wt{\Phi}_N )\triangle \wt{\Phi}_N|$
and $ |\wt{\Phi}_N/2| = r |\wt{\Phi}_N| $. Thus, we get that
\begin{equation*}
\frac{|\big(g+\big({\Phi}_N/2\big)\big) \triangle \big({\Phi}_N/2\big)|}{|{\Phi}_N/2|}
=
\frac{|\big(g+\big(\wt{\Phi}_N/2\big)\big) \triangle \big(\wt{\Phi}_N/2\big)|}{|\wt{\Phi}_N/2|}
\leq 
\frac{ |(2g+ \wt{\Phi}_N )\triangle \wt{\Phi}_N|}{ |\wt{\Phi}_N|},
\end{equation*}
which goes to $0$ as $N \to \infty$, because $(\wt{\Phi}_N)_{N\in \N}$ is a 
\Folner{} sequence in $2G$. This concludes the proof of \eqref{useful_item_2}.
\end{proof}

\begin{remark} \label{remark about shifts of Folner}
We remark here that statements \eqref{item_u_f}, 
\eqref{intersection_is_Folner_2} and \eqref{shifts of Folner} hold in 
any countable abelian group, irrespectively of whether the values 
$r, \ell$ are finite or infinite.
\end{remark}

\begin{lemma} \label{lemma aux folner}
    Let $\Psi=(\Psi_N)_{N\in \N}$ be a \Folner{} sequence in $G$. Then 
    \begin{equation} \label{eq Folner and kernels}
        \lim_{N\to \infty} \frac{|\{g\in \Psi_N : g+ \ker(D) \subset \Psi_N\} |}
    {|\Psi_N|}=1.
    \end{equation}
    In particular, this implies that 
    $\frac{|2\Psi_N|}{|\Psi_N|}=\frac{1}{r} + \oh_{N\to \infty} (1)$.
\end{lemma}
\begin{proof}  
Let $\ker(D)=\{h_1, \dots, h_r\}$. For each $N$ we have that 
$$\{g\in \Psi_N : g+ \ker(D) \not \subset \Psi_N\}
=
\bigcup_{i=1}^r
\{g\in \Psi_N : g+ h_i \notin \Psi_N\} =
\bigcup_{i=1}^r \Psi_N \setminus (\Psi_N - h_i)
\subset \bigcup_{i=1}^r \Psi_N \triangle (\Psi_N - h_i)$$
and therefore
\begin{equation}\label{eq_ker_foln}
\frac{|\{g\in \Psi_N : g+ \ker(D) \not \subset \Psi_N\} |}    {|\Psi_N|}
\leq 
\sum_{i=1}^r 
\frac{|\Psi_N \triangle (\Psi_N - h_i)|}{|\Psi_N|}.
\end{equation}
Since $\Psi$ is a \Folner{} sequence in $G$, each summand in the 
right hand side of \eqref{eq_ker_foln} goes to $0$ as $N\to \infty$, 
so
$$\lim_{N \to \infty} \frac{|\{g\in \Psi_N : g+ \ker(D) \not \subset \Psi_N\} |}    {|\Psi_N|} =0,$$ 
which implies \eqref{eq Folner and kernels}.
Now, for each $N$, if we let $\Phi_N=\{g\in \Psi_N : g+ \ker(D) \subset \Psi_N \} $, then we have that 
$\Phi_N \subset \Psi_N \subset \bigcup_{g\in \Psi_N }
g+ \ker(D),$
which gives that
\begin{equation}\label{F_times_r}
    | 2 \Phi_N| \leq |2\Psi_N| \leq |2\bigcup_{g\in \Psi_N }
g+ \ker(D)|.
\end{equation}

Consider the equivalence relation in $G$ defined by $a\equiv b \iff a-b\in 
\ker(D)$, and for each $g\in G$ let $[g]$ denote the equivalence class of 
$g$. Note that for $g,g'\in G$, $[g]\cap [g'] \neq \emptyset 
\iff [g]=[g']$. 
Let $R_N$ be the number of distinct equivalence classes $[g]$ contained in
$\Phi_N$. Then $|\Phi_N|  = rR_N$ and $|2\Phi_N| = R_N$, 
because $2(g+h_i)=2g$ for each $g\in G$, $h_i\in \ker(D)$, combined with 
the fact that whenever $[g] \cap [g'] = \emptyset$ we also have that $2g \neq 2g'$. 
Hence, $| 2 \Phi_N|=\frac{1}{r}|\Phi_N|$. 

Let also $F_N= \bigcup_{g\in \Psi_N } g+ \ker(D)$, and let $L_N$ be 
the number of different equivalent classes appearing in $\Psi_N$.
As before we see that $| F_N|= r L_N$ and  $|2F_N|= L_N$, so that
$| 2 F_N|=\frac{1}{r}|F_N|$. Moreover, 
$$F_N \setminus \Psi_N \subset \bigcup_{\substack{g\in \Psi_N\\ g+ \ker(D) 
\not \subset \Psi_N}} g+\ker(D), $$
so $|F_N \setminus \Psi_N| \leq r \left|\{g \in \Psi_N: g+ \ker(D) 
\not \subset \Psi_N\} \right|=r|\Psi _N \setminus \Phi_N|=
r(|\Psi _N| -|\Phi_N|)$. Dividing by $|\Psi_N|$, taking $N \to \infty$ 
and using that $ \frac{|\Phi_N|}{|\Psi_N|}\to 1$, we get that 
$\frac{|F_N \setminus \Psi_N|}{|\Psi_N|}\to 0$, so also 
$\frac{|F_N \setminus \Psi_N|}{|F_N|}\to 0$, and since $\Psi_N \subset 
F_N$, we get that $\frac{|\Psi_N|}{|F_N|} \to 1 $. 
Now, dividing by $|\Psi_N|$ in \eqref{F_times_r}, we get that 
$\frac{| 2 \Phi_N|}{|\Psi_N|} \leq \frac{|2\Psi_N|} {|\Psi_N|} \leq 
\frac{|2F_N|}{|\Psi_N|}$. Then $\frac{| 2 \Phi_N|}{|\Psi_N|}= 
\frac{1}{r} \frac{|\Phi_N|}{|\Psi_N|} \to \frac{1}{r}$ and  
$\frac{| 2 F_N|}{|\Psi_N|}= 
\frac{1}{r} \frac{|F_N|}{|\Psi_N|} \to \frac{1}{r}$, which implies that 
$\frac{|2\Psi_N|}{|\Psi_N|} \to \frac{1}{r} $ as $N \to \infty$ and concludes the 
proof of the lemma. 
\end{proof}

As a special case of the above result we also obtain the 
following.

\begin{lemma}\label{lemma aux folner 2}
For any \Folner{} sequence $\Phi=(\Phi_N)_{N\in \N}$ in $G$ we have that 
$\frac{|\Phi_N/2|}{|\Phi_N \cap 2G|}= r + \oh_{N\to \infty} (1) $ 
and $\frac{|\Phi_N/2|}{|\Phi_N|}= \frac{r}{\ell} + \oh_{N\to \infty} (1)$.    
\end{lemma}

\begin{proof}
Applying Lemma \ref{lemma aux folner} with $\Psi_N=\Phi_N/2$ and observing that $2(\Phi_N/2)=\Phi_N \cap 2G$, we get that 
$\frac{|\Phi_N/2|}{|\Phi_N \cap 2G|}= r + \oh_{N\to \infty} (1)$.
Combining the previous with {\cite[Lemma 5.4]{charamaras_mountakis2024}} we get
that $\frac{|\Phi_N/2|}{|\Phi_N|}= \frac{r}{\ell} + \oh_{N\to \infty} (1)$.
\end{proof}

\begin{lemma}\label{lemma_phi/2_phi}
If a \Folner{} sequence 
$\Phi=(\Phi_N)_{N\in \N}$ in $G$ is quasi-invariant with respect to doubling, then 
$(\Phi_N\cap \Phi_N /2)_{N\in \N}$ is also a \Folner{} sequence in $G$.
\end{lemma}

\begin{proof}
From Lemma \ref{lemma aux folner 3} \eqref{useful_item_2} we know that $\Phi/2 = (\Phi_N/2)_{N\in \N}$ 
is a \Folner{} sequence in $G$. Since 
$\Phi$ is q.i.d., it follows from Lemma \ref{lemma aux folner 2} that $\Phi, \Phi/2$ satisfy the assumptions of Lemma
\ref{lemma aux folner 3} \eqref{intersection_is_Folner_2}, and the 
lemma follows. 
\end{proof}

\bibliographystyle{abbrv}
\bibliography{Refs}

\vspace{1cm}

\end{document}